\documentclass[11pt,twoside]{article}
\usepackage{times}
\usepackage{amsmath,amssymb}
\usepackage{color}
\usepackage{amsthm}
\usepackage{tikz}
\usepackage{lineno}

\allowdisplaybreaks

\theoremstyle{plain}
\newtheorem{thm}{Theorem}[section]

\newtheorem{lem}[thm]{Lemma}
\newtheorem{prop}[thm]{Proposition}

\newcommand{\mo}{\mathcal{O}}

\theoremstyle{definition}
\newtheorem{defn}{Definition}[section]
\theoremstyle{remark}
\newtheorem{rem}{Remark}[section]
\theoremstyle{claim}

\numberwithin{equation}{section}
\renewcommand{\theequation}{\thesection.\arabic{equation}}

\def\p{\partial}
\def\ve{\varepsilon}
\def\f{\frac}
\def\na{\nabla}

\def\al{\alpha}

\def\vp{\varphi}

\def\g{\gamma}

\def\o{\omega}
\def\ds{\displaystyle}

\def\mk{\mathcal {K}}
\def\mo{\mathcal {O}}
\def\c{\boldsymbol{c}}

 \allowdisplaybreaks

\begin{document}
 \footskip=0pt
 \footnotesep=2pt
\let\oldsection\section
\renewcommand\section{\setcounter{equation}{0}\oldsection}
\renewcommand\thesection{\arabic{section}}
\renewcommand\theequation{\thesection.\arabic{equation}}
\newtheorem{claim}{\noindent Claim}[section]
\newtheorem{theorem}{\noindent Theorem}[section]
\newtheorem{lemma}{\noindent Lemma}[section]
\newtheorem{proposition}{\noindent Proposition}[section]
\newtheorem{definition}{\noindent Definition}[section]
\newtheorem{remark}{\noindent Remark}[section]
\newtheorem{corollary}{\noindent Corollary}[section]
\newtheorem{example}{\noindent Example}[section]

\title{Global smooth solutions of
3-D null-form wave equations in exterior domains with
Neumann boundary conditions}

\author{Li Jun$^{1,*}$, \qquad Yin Huicheng$^{2, }$\footnote{Li Jun (lijun@nju.edu.cn) and Yin Huicheng
(huicheng@nju.edu.cn, 05407@njnu.edu.cn) are supported by the NSFC (No.11371189, No.11571177),
and the Priority Academic Program Development of Jiangsu Higher Education
Institutions. }\vspace{0.5cm}\\
\small 1.
Department of Mathematics and
IMS, Nanjing University, Nanjing 210093, P.R.China.\\
\small 2.
School of Mathematical Sciences, Nanjing Normal University, Nanjing 210023, P.R.China.\\}
\vspace{0.5cm}

\date{}
\maketitle
\centerline{}
\vskip 0.3 true cm

\centerline {\bf Abstract} \vskip 0.3 true cm

The paper is devoted to investigating long time behavior of smooth small data solutions to 3-D quasilinear wave
equations outside of compact  convex obstacles with Neumann boundary conditions.
Concretely speaking, when the surface of a 3-D compact convex
obstacle is smooth and the quasilinear wave equation fulfills the null condition,
we prove that the smooth small data solution exists globally
provided that the Neumann boundary condition on the exterior domain is given.
The approach is based on the weighted $L^2$ time-space energy estimates associated with modified Klainerman vector fields
which are introduced by Keel-Smith-Sogge \cite{KSS4} and Metcalfe-Sogge \cite{MC2} as well as the classical energy
decay estimates on the first order derivatives of  solutions
for the linear wave equation with Neumann boundary condition by Morawetz \cite{MCS}. One of the main ingredients
in the current paper is the establishment of local energy decay estimates of the
solution itself. As an application of the main result,
the global stability to 3-D static compressible Chaplygin gases in exterior domain is shown
under the initial irrotational perturbation with small amplitude.

\vskip 0.2 true cm

\vskip 0.2 true cm


\section{Introduction}
In this paper, we provide  a rigorous mathematical analysis on the global existence
of smooth small data solutions to 3-D quasilinear wave equations outside of
compact convex obstacles with smooth boundaries. There are two main starting points
for investigating this problem. Mathematically, the problem is one of the fundamental topics in
studying initial-boundary value problems for nonlinear wave equations; Physically,
this problem is closely associated with some mathematical problems arising from aerodynamics,
such as the long time stability of motion of compressible gases in exterior domains. To achieve our
destination of global existence, as in the well-known works done by Christodoulou \cite{CD1}
and S.Klainerman \cite{KS3} for the Cauchy problem of 3-D quasilinear wave equations, we require that the nonlinearities
of the equations involve null forms.

So far the blowup
or global existence of smooth small data solutions have been systematically studied
for the initial data problem of quasilinear wave equations in $\mathbb{R}_{+}\times \Bbb R^n$:\footnote{Hereafter,
the Greek letters $\alpha, \beta, \cdots$ are used to denote an index from $0$ to $n$ with $0$ standing for the time variable and Latin letters $i, j, k, \cdots$ are used to denote an index from $1$ to $n$. Moreover, we use $\partial$ and $\nabla$ to denote the time-space gradient and the space gradient respectively. In addition, repeated indices are always understood as summations.}
\begin{equation}\label{0.1}\begin{cases}
Q^{\alpha\beta}(u, \p
u)\p_{\alpha\beta}^2u=0,\\[2mm]
(u, \p_t u)(0,x)=\ve (u_0, u_1)(x),
\end{cases}
\end{equation}
where $x_0=t$, $x=(x_1, ..., x_n)$, $\ve>0$ is a sufficiently small constant, $u_0(x), u_1(x)$ $\in
C_0^{\infty}(\Bbb R^n)$, $Q^{\alpha\beta}(u, \p u)=Q^{\beta\alpha}(u, \p u)$ are smooth functions in their arguments.
Without loss of generality, one can write
\begin{equation*}\label{0.2}
Q^{\alpha\beta}(u,\p u)=c^{\alpha\beta}+d^{\alpha\beta}u+e^{\alpha\beta}_{\gamma}\p_{\g}u+O(|u|^2+|\p u|^2),
\end{equation*}
where $c^{\alpha\beta}, d^{\alpha\beta}$ and $e^{\alpha\beta}_k$ are constants,
$c^{\alpha\beta}\p_{\alpha\beta}^2=\square\equiv \p_t^2-\Delta$.
By the well-known results in
\cite{HL2}, \cite{KS1}, \cite{KS2}, \cite{LI1} and \cite{LH1}, one has that \eqref{0.1} admits a
global smooth small data solution for $n\ge 4$. If $n=3$, the blowup or global existence of
smooth solutions to \eqref{0.1} have been basically established, one can
see \cite{AS1},  \cite{AS4},  \cite{DLY1}, \cite{DLY2}, \cite{CD1}, \cite{CD2}, \cite{HL2},
\cite{KS3},  \cite{LH2}, \cite{LH3}, \cite{SCD1}, \cite{Speck}
and the references therein. If $n=2$ and  the 2-D nonlinear wave in \eqref{0.1}
admits such a form (whose coefficients depend only
on the first order derivatives $\p u$)
\begin{equation}\label{0.3}\begin{cases}
Q^{\alpha\beta}(\p
u)\p_{\alpha\beta}^2u=0,\\[2mm]
(u, \p_t u)(0,x)=\ve (u_0, u_1)(x),\quad x\in\Bbb R^2,
\end{cases}
\end{equation}
where $Q^{\alpha\beta}(\p u)=Q^{\beta\alpha}(\p u)=c^{\alpha\beta}+e^{\alpha\beta}_{\g}\p_{\g}u
+e^{\alpha\beta}_{\g\mu}\p_{\g}u\p_{\mu}u+O(|\p u|^3)$. It is well-known that
when $e^{\alpha\beta}_{\g}\o_{\al}\o_{\beta}\o_{\g}$ $\not\equiv0$
or $e^{\alpha\beta}_{\g\mu}\o_{\al}\o_{\beta}\o_{\g}\o_{\mu}\not\equiv0$ for
$\o_0=-1$ and $\o=(\o_1, \o_2)\in\Bbb S^1$,
the smooth solution $u$ to \eqref{0.3} blows up in finite time
as long as $(u_0(x), u_1(x))\not\equiv0$ (see \cite{AS1}, \cite{AS3}, \cite{HL2} and so on);
when $e^{\alpha\beta}_{\g}\o_{\al}\o_{\beta}\o_{\g}\equiv0$
and $e^{\alpha\beta}_{\g\mu}\o_{\al}\o_{\beta}\o_{\g}\o_{\mu}\equiv0$,
\eqref{0.3} has a global smooth solution $u$ (see \cite{AS2}).

Compared with Cauchy problems \eqref{0.1} and \eqref{0.3}, Dirichlet-wave equations and Neumann-wave
equations are also attractive for their theoretical significance and applicable values. The Neumann-wave equation
as well
as the Dirichlet-wave equation are two fundamental models in studying initial-boundary value problems of quasilinear wave
equations. For the Dirichlet-wave equations, there have been many interesting works on almost global/global existence of
smooth small data solutions to 2-D/3-D quasilinear scalar equations or systems with multiple speeds
(one can see \cite{DZ},
\cite{KSS1}-\cite{KSS3}, \cite{MC2}-\cite{MC5}, \cite{SCD2} and the references therein).
While, for the Neumann-wave equation, to our best knowledge, so far there have been few
global existence results for the 3-D quasilinear equations in \eqref{0.1} except the symmetric solution case
(see \cite{GP1} and so on). In this paper, we focus on this problem under the convex
condition of compact smooth obstacles.
It is proved that the smooth small data solutions  on the exterior of the convex domain exist globally when
the null conditions are satisfied and some ``admissible condition" is posed. More concretely,
the following initial-boundary value problem of
3-D quasilinear wave equation is considered:
\begin{equation}\label{Mpro}\begin{cases}
\Box u=\mathcal{N}(\partial u, \partial^2 u),\qquad  \  (t, x)\in\mathbb{R}_{+}\times \mathcal{O},\\[2mm]
\partial_{\boldsymbol{\nu}}u=0,\qquad\qquad\qquad \ (t, x)\in \mathbb{R}_{+}\times\partial\mathcal{O},\\[2mm]
(u, \partial_t u)(0, x)=\varepsilon (u_0, u_1)(x),\ \  x\in \mathcal{O},
\end{cases}
\end{equation}
where the nonlinearity $\mathcal{N}(\partial u, \partial^2 u)$ is smooth in its arguments and linear in $\partial^2 u$,
$\mo=\mathbb{R}^3\backslash\mk$ and $\mk$ is a compact convex obstacle with smooth boundary,
$\boldsymbol{\nu}$ stands for the unit outer normal direction of $\mathbb{R}_{+}\times\partial\mathcal{O}$.
Without loss of generality, we assume throughout that $\mk$ contains the origin point and
\begin{equation}\label{CondK}
\mk=\{x=r\omega\in\mathbb{R}^3: r<b(w), \o\in\Bbb S^2\}
\end{equation}
with $b$ being a smooth convex function and  $\frac{3}{4}<b(\omega)<1$.

In addition, $\mathcal{N}(\partial u, \partial^2 u)$ can be written as
\begin{equation}\label{Null1}
\mathcal{N}(\partial u, \partial^2 u)=\mathcal{S}(\partial u)
+\mathcal{Q}^{\alpha\beta}(\partial u)\partial_{\alpha\beta}^2 u
\end{equation}
with
\begin{equation*}
\mathcal{S}(\partial u)=\mathcal{S}^{\alpha\beta}\partial_{\alpha}u\partial_{\beta}u+O(|\partial u|^3),\ \
\mathcal{Q}^{\alpha\beta}(\partial u)=\mathcal{Q}^{\alpha\beta}_{\mu}\partial_{\mu}u+O(|\partial u|^2),\ \
\end{equation*}
where $\mathcal{Q}^{\alpha\beta}(\partial u)=\mathcal{Q}^{\beta\alpha}(\partial u)$, $\mathcal{S}^{\alpha\beta}=\mathcal{S}^{\beta\alpha}$ and $\mathcal{Q}_{\mu}^{\alpha\beta}=\mathcal{Q}_{\mu}^{\beta\alpha}$ are some constants. Meanwhile,
$\mathcal{N}(\partial u, \partial^2 u)$ is assumed to satisfy the null conditions:
\begin{equation}\label{Null2}
\mathcal{S}^{\alpha\beta}\omega_{\alpha}\omega_{\beta}\equiv 0\quad {\text and}
\quad \mathcal{Q}^{\alpha\beta}_{\mu}\omega_{\mu}\omega_{\alpha}\omega_{\beta}\equiv 0
\qquad \text{on}\qquad \omega_0^2=\ds\sum_{i=1}^3\omega_i^2.
\end{equation}
On the other hand, motivated by the fixed wall condition of compressible Chaplygin gases on the exterior domain,
we will pose the following ``admissible condition''
on $\mathcal{N}(\partial u, \partial^2 u)$ for arbitrary smooth functions $v$ and $w$ satisfying
$\partial_{\boldsymbol{\nu}}v|_{\mathbb{R}_{+}\times\partial\mathcal{O}}=0$ and $\partial_{\boldsymbol{\nu}}w|_{\mathbb{R}_{+}\times\partial\mathcal{O}}=0$:
\begin{equation}\label{AC}
\mathcal{Q}^{\alpha\beta}(\partial v)\boldsymbol{\nu}^{\alpha}\partial_{\beta}w\equiv 0,\ \  \ (t, x)\in\mathbb{R}_{+}\times\partial\mathcal{O},
\end{equation}
where $\boldsymbol{\nu}^{\alpha}$ stands for the $\alpha^{th}$ component of $\boldsymbol{\nu}$.
Some explanations on condition \eqref{AC} will be given in Remark \ref{AdC} below.

The goal of the current paper is to find the global smooth small data solution $u$ to \eqref{Mpro}
with conditions \eqref{Null1}-\eqref{AC} and compatibility conditions up to infinite order
for the initial-boundary values.
Here we point out that in order to obtain the smooth solution $u$ of  \eqref{Mpro},
it is necessary to pose the corresponding compatible conditions for the initial-boundary values
(see Remark \ref{Comp}).

To state the main results conveniently, we now introduce some notations and functional spaces throughout the paper.
Let $\|\cdot\|_{L^2(D)}$ stand for the usual $L^2-$norm on domain $D$ and  $\|\cdot\|$
denote $\|\cdot\|_{L^2(\mo)}$.
In addition, $\left<\cdot\right>=:\sqrt{1+|\cdot|^2}$, $A\lesssim B$ means that
there exists a universal positive constant $C$ such that $A\leq CB$.
The main result of this paper is:
\begin{thm}\label{them1} For the given compact convex obstacle $\mk$ in
\eqref{CondK}, if the nonlinearity $\mathcal{N}(\partial u, \partial^2 u)$ satisfies \eqref{Null1}-\eqref{AC} and
the initial smooth data $(u_0(x), u_1(x))$ satisfy the compatibility conditions up to infinite order,
then there is a constant $\varepsilon_0>0$ such that for all $0<\varepsilon\leq \varepsilon_0$,
when
\begin{equation}\label{intialcon}
\sum\limits_{|\beta|\leq 69}\|\left<x\right>^{|\beta|}\nabla^{\beta}u_0\|
+\sum\limits_{|\beta|\leq 68}\|\left<x\right>^{1+|\beta|}\nabla^{\beta} u_1\|\lesssim 1,
\end{equation}
the problem \eqref{Mpro} has a unique global smooth solution $u\in C^{\infty}(\mathbb{R}_{+}\times\mathcal{O})$.
Moreover,
\begin{equation}\label{Decay}
(1+t+|x|)(|u(t, x)|+|\partial u(t, x)|)\lesssim \varepsilon.
\end{equation}
\end{thm}

\begin{rem}\label{AdC} We pose the ``admissible condition'' \eqref{AC} in Theorem \ref{them1} due to the following two reasons:
First, this kind of condition naturally comes from the fixed wall condition of Chaplygin gases on the exterior
domain (see details in Sect.\ref{VII} below). Second, under condition \eqref{AC}
one still obtains the Neumann-type boundary condition for \eqref{Mpro} when the commutator arguments are taken
(see the whole Sect.\ref{III} below).

\end{rem}

\begin{rem}\label{Comp}
We give some illustrations on the compatibility conditions of initial-boundary values
mentioned in Theorem \ref{them1}. Let $J_k u=\{\nabla^{\beta}u: 1\leq |\beta|\leq k\}$
be the collection of all spatial derivatives up to $k$ order of $u$. If $u\in \ds\bigcap_{k=0}^m C^k([0, T], H^{m-k}(\mo))$
$(m\geq 2)$ is a  local solution of \eqref{Mpro}, one then has $\partial_0^k u(0, \cdot)=\psi_k(J_k u_0, J_{k-1}u_1)\ (0\leq k\leq m)$  for certain
functions $\psi_k$ depending on the nonlinearity $\mathcal{N}$. In this case, the $m^{th}$ order compatibility condition
for \eqref{Mpro} is just the requirement that all the $\partial_{\boldsymbol{\nu}}\psi_k$ vanish on $\partial\mathcal{O}$ for
$0\leq k\leq m-1$. In addition, $(u_0, u_1)\in C^{\infty}(\mathcal{O})$ is said to satisfy
the compatibility conditions up to infinite order if $m^{th}$ order compatibility condition holds for all $m\in\Bbb N$.

\end{rem}

\begin{rem}\label{DIRICHLET} For the Dirichlet boundary value problem, some authors have established
the almost global or global existence of smooth small data solutions to quasilinear scalar wave equations
or systems with multiple speeds (see \cite{KSS1}-\cite{KSS3},  \cite{MC2}-\cite{MC5}). The advantage
of homogenous Dirichlet boundary value for wave equations is that the local-in-space $L^2$
norm of solution $u$ itself can be directly derived through the resulting energy estimates of the first derivatives $\p u$
(one can also be referred to \cite{G-0}).
However, for the Neumann-wave equations, the cases are different as pointed
out in \cite{CXY} and \cite{LWY} due to the lack of the value of solution $u$ on the boundary.
Fortunately, for the Neumann boundary value problem \eqref{Mpro}, we can establish the local-in-space $L^2$ norm of
the solution $u$ itself in this paper.

\end{rem}

The proof of Theorem \ref{them1} will follow some key ideas of Klainerman \cite{KS3}
and Metcalfe-Sogge \cite{MC2} associated with the method of Klainerman's vector fields
and the exponential decay estimate of local energy with respect to Neumann-wave equations shown in \cite{MCS}.
The classical Klainerman's vector fields contain the generators of spatial rotations and time-space translations,
the hyperbolic rotations  as well as the scaling vector field, namely,
\begin{equation*}
Z=\{\partial_{\alpha}, \Omega\},\ \  L_i=x_i\partial_0+t\partial_i,\ \ L=t\partial_0+x_{i}\partial_i
\end{equation*}
with $\Omega=\{x_{2}\partial_{3}-x_{3}\partial_{2}, x_{3}\partial_{1}-x_{1}\partial_{3}, x_{1}\partial_{2}-x_{2}\partial_{1}\}$.
As pointed in \cite{CXY}, \cite{G-0}, \cite{MC2} and so on, the hyperbolic rotations $\{L_i\}$ are not applicable to the current
initial-boundary value problems of quasilinear wave equations. This leads to a key difficulty that we can not
use Klainerman's argument in \cite{KS3} to obtain the standard $t^{-1}$ decay of solution $u$ and $t^{-1}\left<r-t\right>^{-1/2}$ decay of $\partial u$. More precisely,
due to the lack of the hyperbolic rotations $\{L_i\}$,
one can not obtain the pointwise decay estimates directly from weighted energy estimates associated with Klainerman's vector fields
by Klainerman's inequality.  Alternatively, as in \cite{MC2}, to compensate the lack of the hyperbolic
rotations $\{L_i\}$, one may use elliptic regularity estimates to improve the spatial regularities of solution
$u$. In this situation, new difficulties can be overcome by establishing the weighted time-space $L^2$ estimates
associated with Klainerman's vector fields $Z$ and $L$ (this can be regarded as the variation version of Keel-Smith-Sogge estimates
in \cite{KSS4}) as well as the pointwise estimates established in Sect.\ref{IV}. From this, one can convert the $\left<x\right>^{-1}$ decay
of solution $u$, obtained from the above estimates, to the expected decay \eqref{Decay} for deriving the global existence of $u$.
Compared with the classical Klainerman' vector field method for treating the initial data problem
of the quasilinear wave equation in \eqref{Mpro} (see \cite{KS3}, \cite{LH3} and so on),
we require to establish  the local-in-space estimate for the solution $u$ itself other than only for $\p u$ as in  \cite{KS3} and \cite{LH3}
since we shall consider $\chi(x)u$ instead of $u$ for some truncated function $\chi(x)$ so that
problem \eqref{Mpro} can be reduced to the boundaryless case. However, unlike the Dirichlet boundary value problem in \cite{MC2},
there is no direct estimate on the solution $u$ itself
for the Neumann boundary value problem due to the lack of the boundary value of $u$.
To overcome this difficulty, on one hand, we establish the exponential-type decay estimate of solution
$u$  in Lemma \ref{Exponetial}; on the other hand, based on the new observation
that  $L^2$ energy of $\p u$ is bounded under the Neumann  boundary value condition (see Lemma \ref{Timenl}),
we will develop a closed elliptic regularity estimate in Lemma \ref{Elliptic} by utilizing
the $L^2$ norm of $\p u$ to replace the $L^2$ norm of $u$ itself. Combining these two key Lemmas
with $L^2$ time-space estimates and pointwise estimates, motivated by \cite{MC2},
we eventually complete the proof of Theorem \ref{them1}.

The paper is organized as follows. In Sect.\ref{II}, some  basic lemmas associated with null conditions,
Sobolev inequality and elliptic regularity estimates are given.
In Sect.\ref{III}-Sect.\ref{IV}, we establish the weighted time-space $L^2$ estimates
associated with partial modified Klainerman's vector fields and further derive the pointwise decay estimates
of solution $u$ to problem \eqref{Mpro}. These estimates will play an essential role in proving Theorem \ref{them1}
by the continuity induction argument. In Sect.\ref{V}, the assumption used in  continuity induction
argument is closed, and then Theorem \ref{them1} is shown. In Sect.\ref{VI}, the local
existence of solution to problem \eqref{Mpro} is established based on the ``admissible condition'' \eqref{AC} and some compatible
conditions. In Sect.\ref{VII}, as an application of Theorem \ref{them1}, we prove the global
stability of the 3-D static compressible Chaplygin gases in an exterior domain under initial irrotational small-amplitude perturbation. In Appendix \ref{A}, we establish two auxiliary lemmas for the application in the proof of Theorem \ref{them1}.

\vskip 0.2 true cm

In addition, throughout the paper, we will use the following notations and functional spaces.

$\bullet$ For the operator vector field $W=(W_1, \cdots, W_l)$ and $k\in (\mathbb{N}\cup\{0\})^{l}$ for any integer $l$, $W^{k}=\prod\limits_{i=1}^{l}W_i^{k_i}$.

$\bullet$ For any positive constants $R_1>R_2>1\ (i=1, 2)$ and $r=|x|=\sqrt{x_1^2+x_2^2+x_3^2}$,
define $\{r\leq R_1\}=\{x\in\mo: r\leq R_1\}$ and $\{R_2<r<R_1\}=\{x\in\mo: R_2<r<R_1\}.$

$\bullet$ Let $\|\cdot\|_{H^k(D)}\ (\|\cdot\|_{L^2(D)})$ stand for the usual $H^k\ (L^2)-$norm on the domain $D$ and $\|\cdot\|_{H^k}\ (\|\cdot\|)$ denote $\|\cdot\|_{H^k(\mo)}\ (\|\cdot\|_{L^2(\mo)})$.

$\bullet$ $C(\cdot)$ is a positive constant only depending on the variable $\cdot$.

$\bullet$ $\mathcal{L}=t\partial_0+\varrho(x)x^i\partial_i$ is the modified operator of $L=t\partial_0+x_i\partial_i$, 
where $\varrho(x)=\varrho(|x|)$ is a smooth radial symmetric function with $\varrho(|x|)=1$ for $|x|\geq 3/2$, 
$\varrho(|x|)=0$ for $|x|\leq 1$ and $0\leq \varrho'(|x|)\leq 2$.

\section{Preliminaries}\label{II}

In this section, some basic results will be listed or established. These results will play a key role
in proving Theorem \ref{them1}.

\begin{lem}\label{Sobolev}{\bf (Sobolev type inequality \cite{KS3})} Suppose that $v\in C^{\infty}(\mo)$, then for $R\geq 4$,
\begin{equation*}\label{sobolev1}
\|v\|_{L^{\infty}(R/2\leq |x|\leq R)}\lesssim R^{-1}\sum\limits_{|\alpha|+|\beta|\leq 2}\|\Omega^{\alpha}\nabla^{\beta}v\|_{L^2(R/4\leq |x|\leq 2R)},
\end{equation*}
and
\begin{equation*}\label{sobolev2}
\|v\|_{L^{\infty}(R\leq |x|\leq R+1)}\lesssim R^{-1}\sum\limits_{|\alpha|+|\beta|\leq 2}\|\Omega^{\alpha}\nabla^{\beta}v\|_{L^2(R-1\leq |x|\leq R+2)}.
\end{equation*}\qed
\end{lem}

\begin{lem}\label{Hardy}{\bf (Hardy type inequality)} For any smooth function $v\in C^{\infty}(\mo)$ decaying fast at infinity, then
\begin{equation}\label{hardy}
\|v/r\| \leq 2\|\partial v\|.
\end{equation}
\end{lem}

\begin{proof} It follows from \eqref{CondK} that
\begin{equation*}\begin{aligned}
\|v/r\|^2&=\int_{\mathbb{S}^2}\int_{b(\omega)}^{+\infty}v^2 (r, \omega)dr d\sigma\\[2mm]
&=-\int_{\mathbb{S}^2}\int_{b(\omega)}^{+\infty}2r (v\partial_r v)(r, \omega)drd\sigma-\int_{\mathbb{S}^2}b(\omega)v^2(b(\omega), \omega)d\sigma\\[2mm]
&\leq 2\|v/r\| \|\partial v\|,
\end{aligned}
\end{equation*}
which yields \eqref{hardy}.\qquad \qquad \qquad \qquad \qquad \qquad \qquad
\qquad \qquad \qquad \qquad \qquad \qquad \qquad \qquad $\square$\end{proof}

\begin{lem}\label{Elliptic}{\bf (Elliptic regularity estimate)} Suppose that $\psi\in C^{\infty}(\mo)$ satisfies $\partial_{\boldsymbol{\nu}}\psi=0$ on $\partial\mo$. Then one has that for any $R\geq 2$ and any nonnegative integer $k$,
\begin{equation}\label{interior}
\|\nabla^2\psi\|_{H^{k}(R\leq |x|\leq R+1)}\leq C(k)(\|\nabla\psi\|_{L^2(R-1\leq |x|\leq R+2)}
+\|\Delta\psi\|_{H^{k}(R-1\leq |x|\leq R+2)}).
\end{equation}
In addition, the following elliptic estimate near boundary $\partial\mo$ holds
\begin{equation}\label{boundary}
\|\nabla^2\psi\|_{H^{k}(|x|\leq 3)}\leq C(k)(\|\nabla\psi\|_{L^2(|x|\leq 4)}
+\|\Delta\psi\|_{H^{k}(|x|\leq 4)}).
\end{equation}
\end{lem}

\begin{proof} At first, we point out that estimate \eqref{interior} does not depend on
the boundary condition of $\psi$ and it just comes from the standard elliptic interior regularity
estimate (see Chapter 6 of \cite{GT}).

As for the estimate \eqref{boundary}, we define $\bar\psi=\ds\frac{1}{|(|x|\leq 4)|}\int_{(|x|\leq 4)}\psi dx$.
Then it follows from the elliptic regularity estimates near the boundary for Neumann boundary conditions (see Theorem 15.2 in \cite{AND})
and Poincare inequality that
\begin{equation*}\begin{aligned}
\|\nabla^2\psi\|_{H^{k}(|x|\leq 3)}&=\|\nabla^2(\psi-\bar\psi)\|_{H^{k}(|x|\leq 3)}\\[2mm]
&\leq C(k)\left(\|\psi-\bar\psi\|_{L^2(|x|\leq 4)}+\|\Delta(\psi-\bar\psi)\|_{H^{k}(|x|\leq 4)}\right)\\[2mm]
&\leq C(k)\left(\|\nabla\psi\|_{L^2(|x|\leq 4)}+\|\Delta \psi\|_{H^{k}(|x|\leq 4)}\right),
\end{aligned}
\end{equation*}
which derives \eqref{boundary}.\qquad \qquad \qquad \qquad \qquad \qquad \qquad \qquad
\qquad \qquad \qquad \qquad \qquad \qquad $\square$\end{proof}

Next we consider the following linear Neumann-wave equation problem
\begin{equation}\label{NW}\begin{cases}
\Box v=F, \ \ (t, x)\in \mathbb{R}_{+}\times\mo,\\[2mm]
\partial_{\boldsymbol{\nu}}v=0,\ \ (t, x)\in\mathbb{R}_{+}\times\partial\mo,\\[2mm]
v(t, x)\equiv 0,\ \ t\leq 0.
\end{cases}
\end{equation}
Then we have that

\begin{lem}\label{Conclusion} For the smooth solution of  problem \eqref{NW}, one has that for any nonnegative  integer $k$,
\begin{equation}\label{nw1}
\sum\limits_{0\leq |\alpha|\leq k}\|\partial^{\alpha}\partial v(t, \cdot)\|\leq C(k)
\left(\sum\limits_{0\leq j\leq k}\int_0^{t}\|\partial_0^j F(s, \cdot)\|ds
+\sum\limits_{0\leq |\alpha|\leq k-1}\|\partial^{\alpha}  F(t, \cdot)\|\right).
\end{equation}
\end{lem}

\begin{proof} For any nonnegative integer $j$, by $\partial_{\boldsymbol{\nu}}\partial_0^j v=0$ on $\mathbb{R}_{+}\times\partial\mo$
and the standard energy estimate (see Lemma \ref{Timenl}), we get
\begin{equation}\label{see}
\|\partial \partial_0^j v(t, \cdot)\|\lesssim \int_0^{t}\|\partial_0^{j}F(s, \cdot)\|ds.
\end{equation}
In addition, by Lemma \ref{Elliptic}, one has that for $0\le j\le k-1$,
\begin{equation*}\label{see1}
\|\partial^2\partial_0^j v(t, \cdot)\|_{k-j-1}\lesssim \|\partial \partial_0^{j}v(t, \cdot)\|
+\|\partial_0^{j}\Box v(t, \cdot)\|_{k-j-1}+\|\partial_0^{j+2}v(t, \cdot)\|_{k-j-1}.
\end{equation*}
Then by induction argument on $j$ from $0$ to $k-1$, we have
\begin{equation}\label{see2}
\sum\limits_{0\leq |\alpha|\leq k}\|\partial^{\alpha}\partial v(t, \cdot)\|\leq C(k)\left(\sum\limits_{0\leq j\leq k}\|\partial\partial_0^{j} v(t, \cdot)\|+\sum\limits_{0\leq |\alpha|\leq k-1}\|\partial^{\alpha} F(t, \cdot)\|\right).
\end{equation}
Thus, \eqref{nw1} is obtained from \eqref{see} and \eqref{see2} by induction argument on $j$ from $0$ to $k-1$.
\qquad $\square$\end{proof}

\begin{lem}\label{Exponetial} {\bf (Exponential-type decay estimate)}
For problem \eqref{NW}, suppose that assumption
\eqref{CondK} holds, and $F$ is smooth and supported in $\{|x|\leq R\}$ for some positive $R>1$.
Then there exist two positive constants $\kappa$ and $\c<1$ depending on $R$ and the geometry of $\partial\mo$
such that for the smooth solution $v$ of \eqref{NW},
\begin{equation}\label{exponential}
\|\partial v(t, \cdot)\|_{L^2(|x|\leq 5)}\leq \kappa \int_0^{t}e^{-\c(t-s)}\|F(s, \cdot)\|ds,
\end{equation}
and
\begin{equation}\label{exponential1}
\|v(t, \cdot)\|_{L^2(|x|\leq 4)}\leq \kappa\sum\limits_{0\leq j\leq 1}\int_0^{t}e^{-\c(t-s)}\|\partial_0^j F(s, \cdot)\|ds.
\end{equation}
\end{lem}

\begin{proof} Note that estimate \eqref{exponential}
directly comes from THE MAIN THEOREM of \cite{MCS} and the Duhamel principle.

In addition, by $\partial_{\boldsymbol{\nu}}\partial_0 v=0$ on $\mathbb{R}_{+}\times\partial\mo$,
it follows from \eqref{exponential} with $\partial_0 v$ that
\begin{equation*}\label{decayt}
\sum\limits_{0\leq j\leq 1}\|\partial\partial_0^j v(t, \cdot)\|_{L^2(|x|\leq 5)}\lesssim\sum\limits_{0\leq j\leq 1}\int_0^{t}e^{-\c (t-s)}\|\partial_0^j F(s, \cdot)\|ds.
\end{equation*}
Combining this with \eqref{see2} in Lemma \ref{Conclusion}  for $k=1$ that
\begin{equation}\label{decayt1}
\sum\limits_{0\leq |\alpha|\leq 1}\|\partial^{\alpha}\partial v(t, \cdot)\|\lesssim\left(\sum\limits_{0\leq j\leq 1}\int_0^{t}e^{-\c(t-s)}\|\partial_0^j F(s, \cdot)\|ds+\|F(t, \cdot)\|\right).
\end{equation}
Based on \eqref{decayt1}, we now derive \eqref{exponential1}. For this purpose,
at first we take a coordinate transformation such that $\partial\mo$ is changed
into $\mathbb{S}^2$. Let
\begin{equation}\label{cr}
y=\frac{x}{(1-\varrho(x/2))b(\omega)+\varrho(x/2)}.
\end{equation}
Obviously, one has
\begin{equation}\label{Inverse}\begin{aligned}
&\partial_r\left(\frac{r}{(1-\varrho(r/2))b(\omega)+\varrho(r/2)}\right)\\[2mm]
=&\frac{b(\omega)+(1-b(\omega))\varrho(r/2)-1/2(1-b(\omega))\varrho'(r/2)}{((1-\varrho(r/2))b(\omega)+\varrho(r/2))^2}\\[2mm]
>&0,
\end{aligned}
\end{equation}
where the last inequality comes from the assumption of $b(\omega)$ in \eqref{CondK} and the definition of $\varrho$ in the end of Sect.1.\footnote{If $b(\omega)$ in \eqref{CondK} is defined to fulfill $K<b(\omega)<1$ for some positive constant $K$, then \eqref{Inverse} also stands if we select the function $\varrho$ such that  $|\varrho'(r)|\leq \frac{2K}{1-K}$.}

With \eqref{Inverse} and the definition of $\varrho$, the transformation \eqref{cr} is a diffeomorphism from $\mo$ to $\mathbb{R}^3\backslash B_1(0)$. In addition, it is easy to know that the Jacobian matrices $\partial y/\partial x$ and $\partial x/\partial y$
satisfy for any nonnegative integer $k$,
\begin{equation}\label{cr1}
\|\partial y/\partial x-I_3\|_{C_0^k(|x|\leq 3)}+\|\partial x/\partial y-I_3\|_{C_0^k(1\leq |y|\leq 3)}\leq C(k),\
\end{equation}
where $I_3$ represents the $3\times 3$ unit matrix.
Under transformation \eqref{cr}, with the help of \eqref{cr1}, problem \eqref{NW} is converted into
\begin{equation}\label{NW1}\begin{cases}
\Box  v=F+\chi_{(|y|\leq 3)}(y)L_1(\partial v, \partial^2 v),\qquad (t, y)\in \mathbb{R}_{+}\times \mathbb{R}^3\backslash B_1(0),\\[2mm]
\partial_r v=L_2(\partial v),\qquad (t, y)\in\mathbb{R}_{+}\times\partial B_1(0),\\[2mm]
v(t, y)\equiv 0,\qquad t<0,
\end{cases}
\end{equation}
where $\chi_{(|y|\leq 3)}$ is the characteristic function of domain $\{y: 1\leq |y|\leq 3\}$, $L_1(\partial v, \partial^2 v)$
and $L_2(\partial v)$ are linear with respect to their arguments, and satisfies
$$|L_1(\partial v, \partial^2 v)|\lesssim |\partial v|+|\partial^2 v|\quad\text{and}\quad
|L_2(\partial v)|\lesssim |\partial v|.$$
Set $\ds V=\int_{\mathbb{S}^2}v d\sigma/|\mathbb{S}^2|$, one has from \eqref{NW1} that
\begin{equation}\label{NW2}\begin{cases}
\partial_t^2 V-\partial_r^2 V-\frac{2}{r}\partial_r V=F_1,\qquad (t, r)\in\mathbb{R}_{+}\times (1, +\infty),\\[2mm]
\partial_r V=F_2,\qquad (t, r)\in\mathbb{R}_{+}\times\{1\},\\[2mm]
V(t, r)\equiv 0,\qquad t<0,
\end{cases}
\end{equation}
where $\ds F_1=\int_{\mathbb{S}^2}\Box v d\sigma/|\mathbb{S}^2|$ and $\ds F_2=\int_{\mathbb{S}^2}\partial_r v d\sigma/|\mathbb{S}^2|$.
It follows from \eqref{NW2} and the characteristics method that $V$ admits the following explicit expression
\begin{equation}\label{EE}
V(t, r)=\frac{1}{r}V_0(t-r+1)+\frac{1}{r}\int_1^{r}\left[\int_{r'}^{t-r+2r'}sF_1(t-r+2r'-s, s)ds\right]dr',
\end{equation}
where $V_0(p)$ is expressed as
\begin{equation}\label{EE1}
V_0(p)=e^{-p}\int_0^{p}e^{l}\left[\int_{1}^{l+1}sF_1(l+1-s, s)ds-F_2(l,1)\right]dl.
\end{equation}
By \eqref{EE}, Poincare inequality and \eqref{cr}-\eqref{cr1}, we arrive at
\begin{equation}\label{EE2}\begin{aligned}
\|v\|_{L^2(|x|\leq 4)}&\lesssim \|v-V(t,1)\|_{L^2(1\leq |y|\leq 4)}+|V(t, 1)|\\[2mm]
&\lesssim \|\partial v\|_{L^2(1\leq |y|\leq 4)}+|V_0(t)|\\[2mm]
&\lesssim\|\partial v\|_{L^2(|x|\leq 4)}+\int_0^{t}e^{-(t-s)}\|F_1(s, \cdot)\|ds+e^{-t}\int_0^{t}|F_2(l, 1)|dl.
\end{aligned}
\end{equation}
With the definitions of $F_i\ (i=1,2)$ in \eqref{NW1} and \eqref{NW2}, it derives from \eqref{EE1}-\eqref{EE2} and Sobolev trace theorem that
\begin{equation*}\label{EE3}\begin{aligned}
\|v\|_{L^2(|x|\leq 4)}&\lesssim \|\partial v\|_{L^2(|x|\leq 4)}+\int_0^{t}e^{-(t-s)}\|F(s, \cdot)\|ds\\[2mm]
&+\int_0^{t}e^{-(t-s)}\|(\partial v, \partial^2 v)(s, \cdot)\|_{L^2(|x|\leq 4)}ds+e^{-t}\int_0^{t}\|(\partial v, \partial^2 v)(s, \cdot)\|_{L^2(|x|\leq 4)}ds.
\end{aligned}
\end{equation*}
Combining this with \eqref{decayt1} yields \eqref{exponential1}.\qquad
\qquad \qquad \qquad \qquad \qquad \qquad \qquad \qquad \qquad $\square$
\end{proof}

\begin{rem}\label{Ipdecay} Under the assumptions in Lemma \ref{Exponetial}, similar to \eqref{decayt1},
estimate \eqref{nw1} can be improved as
\begin{equation*}
\sum\limits_{0\leq |\alpha|\leq k}\|\partial^{\alpha}\partial v(t, \cdot)\|\leq C(k)\left(\sum\limits_{0\leq j\leq k}\int_0^{t}e^{-\c(t-s)}\|\partial_0^j F(s, \cdot)\|ds+\sum\limits_{0\leq |\alpha|\leq k-1}\|\partial^{\alpha}F(t, \cdot)\|\right).
\end{equation*}
\end{rem}

Next two lemmas show the relations between null conditions in \eqref{Null2} and the partial
Klainerman's vector fields $\Gamma=\{Z, L\}$, which can be found in \cite{KSS4} and \cite{MC2}.

\begin{lem}\label{Null}{\bf (Estimates under null conditions)} Suppose that the null conditions \eqref{Null2} hold. Then
\begin{equation*}\label{null2}
|\mathcal{S}^{\alpha\beta}\partial_{\alpha}w\partial_{\beta}v|\lesssim \left<r\right>^{-1}(|\Gamma w| |\partial v|+|\partial w| |\Gamma v|)
\end{equation*}
and
\begin{equation*}\label{null1}
|\mathcal{Q}_{\mu}^{\alpha\beta}\partial_{\mu}w\partial_{\alpha\beta}^2 v|\lesssim \left<r\right>^{-1}(|\Gamma w||\partial^2 v|+|\partial w| |\partial \Gamma v|)+\frac{\left<t-r\right>}{\left<t+r\right>}|\partial w| |\partial^2 v|.
\end{equation*}
\qed
\end{lem}

\begin{lem}\label{Weighted}{\bf (Weighted energy estimates \cite{KSS4})} Suppose that $v\in C^{\infty}(\mathbb{R}\times\mathbb{R}^3)$
vanishes for large $x$  and $v(t, x)\equiv 0$ for $t\leq 0$, then one has that for any given nonnegative integers $\mu$ and $\alpha$,
\begin{equation*}\label{weighted}\begin{aligned}
(\ln(2+T))^{-1/2}\|\left<x\right>^{-1/2}\partial(L^{\mu} Z^{\alpha}v)\|_{L^2([0, T]\times\mathbb{R}^3)}
\lesssim\int_0^{T}\|\Box L^{\mu}Z^{\alpha}v(s, \cdot)\|_{L^2(\mathbb{R}^3)} ds.
\end{aligned}
\end{equation*}
In particular,
\begin{equation}\label{weighted1}
\sum\limits_{|\beta|\leq 1}\|L^{\mu}Z^{\alpha}\partial^{\beta}v\|_{L^2([0, T]\times (|x|\leq 4))}
\lesssim \int_0^{T}\sum\limits_{\nu\leq \mu}\|\Box L^{\nu}Z^{\alpha}v(s, \cdot)\|_{L^2(\mathbb{R}^3)}ds.
\end{equation}
\qed
\end{lem}

\section{$L^2$ spatial-energy and time-space energy estimates}\label{III}

In this section, we will focus on some fundamental energy estimates of  smooth solutions $u$ to the following
linear Neumann-wave equation with variable coefficients
\begin{equation}\label{VLNV}\begin{cases}
\Box_{h}u=F,\ \ (t,x)\in\mathbb{R}_{+}\times\mo,\\[2mm]
\partial_{\boldsymbol{\nu}}u=0,\ \ (t, x)\in\mathbb{R}_{+}\times\partial\mo,\\[2mm]
(u, \partial_t u)(t, x)=(f, g)(x),\ \ x\in\mo,
\end{cases}
\end{equation}
where $\Box_h=\Box+h^{\alpha\beta}(t,x)\partial_{\alpha\beta}^2$ with $h^{\alpha\beta}=h^{\beta\alpha}$.
To simplify notations, we set
\begin{equation}\label{notation}
\|h(t)\|=\sum\limits_{\alpha, \beta=0}^{3}\|h^{\alpha\beta}(t, \cdot)\|_{L^{\infty}},\ \ \  \|\partial h(t)\|=\sum\limits_{\alpha, \beta=0}^{3}\|\partial h^{\alpha\beta}(t, \cdot)\|_{L^{\infty}}.
\end{equation}
In addition, $h^{\alpha\beta}$ is always assumed to satisfy the ``admissible condition'' for any  smooth function
$w$ with $\p_{\nu}w|_{\mathbb{R}_{+}\times\p\mo}=0$
\begin{equation}\label{AC1}
h^{\alpha\beta}\boldsymbol{\nu}^{\alpha}\partial_{\beta}w=0,\ \  \ (t, x)\in\mathbb{R}_{+}\times\partial\mo.
\end{equation}
Furthermore, we assume that $h^{\alpha\beta}$ satisfies for some suitably small $\delta>0$,
\begin{equation}\label{size}
\|h(t)\|\leq\frac{\delta}{1+t}.
\end{equation}
Define the energy forms $\{e_{\alpha}(u)\}$ associated with $\Box_h$ as follows
\begin{equation}\label{BEF}
e_0(u)=|\partial u|^2+2h^{0\beta}\partial_0 u\partial_{\beta}u
-h^{\alpha\beta}\partial_{\alpha}u\partial_{\beta}u, \ \ e_i(u)=-2\partial_0 u\partial_i u+2h^{i\beta}\partial_0 u\partial_{\beta}u.
\end{equation}
In addition, we define the following energy functionals for any nonnegative integers $\mu$ and $\nu$,
\begin{equation}\label{BE}
\mathcal{E}_{\mu, \nu}(t)=\sum\limits_{i\leq \mu,\  j\leq\nu}\int_{\mo} e_0(\mathcal{L}^{i}\partial_0^j u)(t, x)dx,
\ \ E_{\mu, \nu}(t)=\sum\limits_{i\leq \mu,\  |\alpha|\leq\nu}\int_{\mo}e_0(L^{i}Z^{\alpha}u)(t,x)dx.
\end{equation}
Under assumption \eqref{size} and the smallness of $\delta>0$ (for instance, $\delta<\frac{1}{40}$), $e_0(u)$ in \eqref{BEF} satisfies
\begin{equation}\label{elcompare}
1/2 |\partial u|^2<e_0(u)<2 |\partial u|^2.
\end{equation}

\subsection{Spatial energy estimates}\label{III.1}

We now establish some spatial energy estimates for operator $\Box_h$.

\begin{lem}\label{Timenl} Under the ``admissible condition'' \eqref{AC1}, if $u$ is
a smooth solution of \eqref{VLNV} which vanishes for large $x$,  then we have that
for any nonnegative integers $\mu$ and $\nu$,
\begin{equation}\label{timenl}
\partial_0 \mathcal{E}_{\mu, \nu}^{1/2}(t)\lesssim \sum\limits_{i\leq \mu;\ j\leq \nu}\|\Box_h(\mathcal{L}^{i}\partial_0^j u)(t, \cdot)\|+\|\partial h(t)\| \mathcal{E}_{\mu, \nu}^{1/2}(t).
\end{equation}
\end{lem}

\begin{proof} We only give the proof of \eqref{timenl} for $\mu, \nu=0$ since the other
cases for general $\mu, \nu$ can be done in the same way with the fact that $\partial_{\boldsymbol{\nu}}\mathcal{L}^{i}\partial_0^j u=0$ on $\mathbb{R}_{+}\times\partial\mo$ for $0\leq i\leq \mu$ and $0\leq j\leq \nu$.

It follows from a direct computation that
\begin{equation*}
2\Box_h u \cdot\partial_0 u=\partial_{\alpha}e_{\alpha}(u)-2\partial_{\alpha}h^{\alpha\beta}\partial_0 u\partial_{\beta}u+\partial_0 h^{\alpha\beta}\partial_{\alpha}u\partial_{\beta}u.
\end{equation*}
Integrating this identity over $\mo$ and using \eqref{elcompare} yield
\begin{equation}\label{basicjf}\begin{aligned}
\partial_0 \mathcal{E}_{0, 0}(t)&\lesssim \|\Box_h u(t, \cdot)\| \|\partial_0 u(t, \cdot)\|+\|\partial h(t)\|\|\partial u(t, \cdot)\|^2-\int_{\mo}e_i(u)\boldsymbol{\nu}^i d\sigma\\[2mm]
&\lesssim \|\Box_h u(t, \cdot)\| \mathcal{E}_{0, 0}^{1/2}(t)+\|\partial h(t)\| \mathcal{E}_{0, 0}(t)-\int_{\mo}e_i(u)\boldsymbol{\nu}^i d\sigma.
\end{aligned}
\end{equation}
Note that $\ds\int_{\mo}e_i(u)\boldsymbol{\nu}^i d\sigma=0$ holds due to the Neumann boundary condition
$\partial_{\boldsymbol{\nu}}u=0$ on $\mathbb{R}_{+}\times\partial\mo$ and \eqref{AC1}. Then
estimate \eqref{timenl} comes from \eqref{basicjf} directly.\qquad \qquad \qquad \qquad \qquad
\qquad \qquad \qquad $\square$
\end{proof}

Let commutator $[O, P]=OP-PO$ stand for two differential operators $O$ and $P$. Direct computation yields
\begin{equation}\label{Chang1}
|\Box_h (\mathcal{L}^i\partial_0^j u)|\leq |\mathcal{L}^i \partial_0^j\Box_h u|+|[\mathcal{L}^i, \Box]\partial_0^j u|+|[\mathcal{L}^i\partial_0^j, h^{\alpha\beta}\partial_{\alpha\beta}^2]u|
\end{equation}
and
\begin{equation}\label{Chang2}
|[\mathcal{L}^i, \Box]\partial_0^j u|\lesssim\sum\limits_{k\leq i-1}|L^k\partial_0^j\Box u|
+\chi_{(|x|\leq 2)}(x)|\sum\limits_{k+|\alpha|\leq i+j,\ k\leq i-1}|L^k\partial^{\alpha}\partial u|,
\end{equation}
here $\chi_{(|x|\leq 2)}(x)$ is the characteristic function of the domain $\{|x|\leq 2\}$.
Combining \eqref{Chang1}-\eqref{Chang2} with \eqref{timenl} in Lemma \ref{Timenl} shows:
\begin{lem} \label{Scaling}Under the assumptions in Lemma \ref{Timenl}, for any nonnegative integers $\mu,  \nu$, one has
\begin{equation}\label{scalnl}\begin{aligned}
\partial_0 \mathcal{E}_{\mu, \nu}^{1/2}(t)&\lesssim \sum\limits_{i\leq \mu;\ j\leq \nu}\|\mathcal{L}^{i}\partial_0^j\Box_h u(t, \cdot)\| +\|\partial h(t)\| \mathcal{E}_{\mu, \nu}^{1/2}(t)\\[2mm]
&\quad + \sum\limits_{i\leq \mu,\ j\leq\nu}\|[\mathcal{L}^{i}\partial_0^{j}, h^{\alpha\beta}\partial_{\alpha\beta}^2]u(t, \cdot)\|
+\sum\limits_{i\leq \mu-1,\ j\leq\nu}\|L^{i}\partial_0^{j}\Box u(t, \cdot)\|\\[2mm]
&\quad +\sum\limits_{k+|\alpha|\leq \mu+\nu\atop k\leq \mu-1}\|L^{k}\partial^{\alpha}\partial u(t, \cdot)\|_{L^2(|x|\leq 2)}.
\end{aligned}
\end{equation}
\end{lem}
Next we establish a higher order energy estimate for the wave operator.
\begin{lem}\label{Spacenl} For fixed integers $N_0$ and $\mu_0$, one has
\begin{equation}\label{spacenl}
\sum\limits_{|\alpha|\leq N_0}\|L^{\mu_0}\partial^{\alpha}\partial u(t, \cdot)\|\lesssim \sum\limits_{\mu+\beta\leq \mu_0+N_0\atop \mu\leq \mu_0}\|L^{\mu}\partial_0^{\beta}\partial u(t, \cdot)\|+\sum\limits_{\mu+|\alpha|\leq N_0+\mu_0-1\atop \mu\leq \mu_0}\|L^{\mu}\partial^{\alpha}\Box u(t, \cdot)\|.
\end{equation}

\end{lem}

\begin{proof} At first, it follows from \eqref{see2} in Lemma \ref{Conclusion} that
\begin{equation}\label{spacenl1}
\sum\limits_{|\alpha|\leq N_0}\|\partial^{\alpha}\partial u(t, \cdot)\|\lesssim \sum\limits_{\beta\leq N_0}\|\partial_0^{\beta}\partial u(t, \cdot)\|+\sum\limits_{|\alpha|\leq N_0-1}\|\partial^{\alpha}\Box u(t, \cdot)\|.
\end{equation}
This shows \eqref{spacenl} for $\mu_0=0$.

Next, we prove \eqref{spacenl} for general $\mu_0$. Since $[\Box, L]=2\Box$, with the help of \eqref{interior} in Lemma \ref{Elliptic}, one has
\begin{equation}\label{spacenl3}
\sum\limits_{|\alpha|\leq N_0}\|L^{\mu_0}\partial^{\alpha}\partial u(t, \cdot)\|_{L^2(|x|\geq 2)}\lesssim \sum\limits_{\mu+\beta\leq \mu_0+N_0\atop \mu\leq \mu_0}\|L^{\mu}\partial_0^{\beta}\partial u(t, \cdot)\|+\sum\limits_{\mu+|\alpha|\leq N_0+\mu_0-1\atop \mu\leq\mu_0}\|L^{\mu}\partial^{\alpha}\Box u(t, \cdot)\|.
\end{equation}
In addition,  by \eqref{boundary} in Lemma \ref{Elliptic}, we arrive at
\begin{equation}\label{spacenl4}\begin{aligned}
&\sum\limits_{|\alpha|\leq N_0}\|L^{\mu_0}\partial^{\alpha}\partial u(t, \cdot)\|_{L^2(|x|\leq 2)}\\[2mm]
\lesssim& \sum\limits_{\mu+\beta\leq N_0+\mu_0\atop \mu\leq \mu_0}t^{\mu}
\|\partial_0^{\mu}\partial^{\beta}\partial u(t, \cdot)\|_{L^2(|x|\leq 2)}\\[2mm]
\lesssim& \sum\limits_{\mu+\beta\leq N_0+\mu_0\atop \mu\leq \mu_0}t^{\mu}\|\partial_0^{\mu+\beta}\partial u(t, \cdot)\|_{L^2(|x|\leq 4)}+\sum\limits_{\mu+\beta\leq N_0+\mu_0-1\atop \mu\leq \mu_0} t^{\mu}
\|\partial^{\beta}\partial_0^{\mu}\Box u(t, \cdot)\|_{L^2(|x|\leq 4)}\\[2mm]
\lesssim&\sum\limits_{\beta\leq N_0}\|L^{\mu_0}\partial_0^{\beta}\partial u(t, \cdot)\|+\sum\limits_{\mu+|\alpha|\leq N_0+\mu_0\atop \mu\leq \mu_0-1}\|L^{\mu}\partial^{\alpha}\partial u(t, \cdot)\|+\sum\limits_{\mu+|\alpha|\leq N_0+\mu_0-1\atop \mu\leq \mu_0}\|L^{\mu}\partial^{\alpha}\Box u(t, \cdot)\|.
\end{aligned}
\end{equation}
Combining \eqref{spacenl4} with \eqref{spacenl1} and \eqref{spacenl3} yields \eqref{spacenl} by induction method on $\mu_0$.
\qquad \qquad $\square$
\end{proof}

\begin{prop}\label{Integral} Assume that $\|\partial h(t)\|\leq\frac{\delta}{1+t}$ holds for small $\delta>0$.
In addition, suppose
\begin{equation}\label{spacenl5}\begin{aligned}
&\sum\limits_{\mu+\beta\leq N_0+\mu_0\atop \mu\leq \mu_0}\left(\|\mathcal{L}^{\mu}\partial_0^{\beta}\Box_h u(t, \cdot)\|
+\|[\mathcal{L}^{\mu}\partial_0^{\beta}, h^{\alpha\beta}\partial_{\alpha\beta}^2]u(t, \cdot)\|\right)\\[2mm]
&\leq\frac{\delta}{1+t}\sum\limits_{\mu+\beta\leq N_0+\mu_0\atop \mu\leq \mu_0}\|\mathcal{L}^{\mu}\partial_0^{\beta}\partial u(t, \cdot)\|+H_{\mu_0, N_0}(t)
\end{aligned}
\end{equation}
for  fixed numbers $N_0$ and $\mu_0$, and a positive function $H_{\mu_0, N_0}(t)$. Then there exists a positive constant $A$
such that
\begin{equation}\label{spacenl6}\begin{aligned}
&\sum\limits_{\nu+|\alpha|\leq N_0+\mu_0\atop \mu\leq \mu_0}\|L^{\mu}\partial^{\alpha}\partial u(t, \cdot)\|\\[2mm]
\lesssim& \sum\limits_{\mu+\alpha\leq N_0+\mu_0-1\atop \mu\leq \mu_0}\|L^{\mu}\partial^{\alpha}\Box u(t, \cdot)\|+(1+t)^{A\delta}\mathcal{E}_{\mu, \nu}^{1/2}(0)\\[2mm]
&+(1+t)^{A\delta}\int_0^{t}\sum\limits_{\mu+|\alpha|\leq N_0+\mu_0-1\atop \mu\leq \mu_0-1}\|L^{\mu}\partial^{\alpha}\Box u(s, \cdot)\| ds+(1+t)^{A\delta}\int_0^{t}H_{\mu_0, N_0}(s)ds\\[2mm]
&+(1+t)^{A\delta}\int_0^{t}\sum\limits_{\mu+|\alpha|\leq N_0+\mu_0\atop \mu\leq \mu_0-1}
\|L^{\mu}\partial^{\alpha}\partial u(s, \cdot)\|_{L^2 (|x|\leq 2)}ds.
\end{aligned}
\end{equation}
\end{prop}

\begin{proof}
It follows from \eqref{scalnl} in Lemma \ref{Scaling} and the assumptions in Proposition \ref{Integral} that
\begin{equation}\label{3.0}\begin{aligned}
\partial_0 \mathcal{E}_{\mu_0, N_0}^{1/2}(t)&\leq \frac{A\delta}{1+t}\mathcal{E}_{\mu_0, N_0}^{1/2}(t)
+AH_{\mu_0, N_0}(t)+A\sum\limits_{\mu+\beta\leq N_0+\mu_0-1\atop \mu\leq \mu_0-1}\|L^{\mu}\partial_0^{\beta}\Box u(t, \cdot)\|\\[2mm]
&+A\sum\limits_{\mu+|\alpha|\leq N_0+\mu_0\atop \mu\leq \mu_0-1}\|L^{\mu}\partial^{\alpha}\partial u(t, \cdot)\|_{L^2(|x|\leq 2)},
\end{aligned}
\end{equation}
where $A>0$ is a generic constant. From \eqref{3.0}, $\mathcal{E}_{\mu_0, N_0}^{1/2}$ can be controlled by the right
hand side of \eqref{spacenl6} via using Gronwall's inequality.

On the other hand, it is derived from \eqref{spacenl} in Lemma \ref{Spacenl} and the definition of $\mathcal{L}$ that
\begin{equation}\label{Convert1}\begin{aligned}
&\sum\limits_{\mu+\beta\leq N_0+\mu_0\atop \mu\leq \mu_0}\|L^{\mu}\partial^{\alpha}\partial u(t, \cdot)\|\\[2mm]
\lesssim&\sum\limits_{\mu+\beta\leq N_0+\mu_0\atop\mu\leq \mu_0}\|L^{\mu}\partial_0^{\beta}\partial u(t, \cdot)\|+\sum\limits_{\mu+|\alpha|\leq N_0+\mu_0-1\atop \mu\leq \mu_0}\|L^{\mu}\partial^{\alpha}\Box u(t, \cdot)\|\\[2mm]
\lesssim&\sum\limits_{\nu+|\alpha|\leq N_0+\mu_0\atop \mu\leq \mu_0}\|\mathcal{L}^{\mu}\partial^{\alpha}\partial u(t, \cdot)\|+\sum\limits_{\mu+|\alpha|\leq N_0+\mu_0-1\atop \mu\leq \mu_0}\|L^{\mu}\partial^{\alpha}\Box u(t, \cdot)\|.
\end{aligned}\end{equation}
Meanwhile, it follows from \eqref{BE}  and \eqref{elcompare} that
\begin{equation*}\label{Convert}
\sum\limits_{\mu+\beta\leq N_0+\mu_0\atop \mu\leq \mu_0}\|\mathcal{L}^{\mu}\partial_0^{\beta}\partial u(t, \cdot)\|\leq 2\mathcal{E}_{\mu_0, N_0}^{1/2}(t).
\end{equation*}
Combining this with \eqref{Convert1} and the mentioned estimate for $\mathcal{E}_{\mu_0, N_0}^{1/2}(t)$ yields \eqref{spacenl6}.
\qquad\qquad \qquad $\square$\end{proof}

We now derive the estimate of $E_{\mu, \nu}$ defined in \eqref{BE}.
\begin{lem}\label{Rotationnl} Suppose that $\|\partial h(t)\|\leq\frac{\delta}{1+t}$ holds for  sufficiently small
$\delta>0$, then
\begin{equation}\label{rotationnl}\begin{aligned}
\partial_0 E_{\mu_0, N_0}(t)&\lesssim \sum\limits_{\mu+|\alpha|\leq N_0+\mu_0\atop \mu\leq \mu_0}\|\Box_h L^{\mu}Z^{\alpha}u(t, \cdot)\| E_{\mu_0, N_0}^{1/2}\\[2mm]
&\quad +\|\partial h(t)\| E_{\mu_0, N_0}
+\sum\limits_{\mu+|\alpha\leq N_0+\mu_0+1\atop \mu\leq \mu_0}\|L^{\mu}\partial^{\alpha}\partial u(t, \cdot)\|_{L^2(|x|\leq 2)}^2.
\end{aligned}
\end{equation}
\end{lem}
\begin{proof} By the analogous argument in Lemma \ref{Timenl}, we have that
\begin{equation}\label{rotationnl1}\begin{aligned}
\partial_0 E_{\mu_0, N_0}(t)&\lesssim \sum\limits_{\mu+|\alpha|\leq N_0+\mu_0\atop \mu\leq \mu_0}\|\Box_h L^{\mu}Z^{\alpha}u(t, \cdot)\| E_{\mu_0, N_0}^{1/2}\\[2mm]
&\quad +\|\partial h(t)\| E_{\mu_0, N_0}+\sum\limits_{\mu+|\alpha|\leq N_0+\mu_0\atop \mu\leq \mu_0}\int_{\partial\mo}|e_j(L^{\mu} Z^{\alpha}u)\boldsymbol{\nu}^{j}|d\sigma.
\end{aligned}
\end{equation}
In addition, it follows from Sobolev trace theory that
\begin{equation*}\begin{aligned}
\sum\limits_{\mu+|\alpha|\leq N_0+\mu_0\atop \mu\leq \mu_0}\int_{\partial\mo}|e_j(L^{\mu} Z^{\alpha}u)\boldsymbol{\nu}^{j}|d\sigma
\lesssim&\sum\limits_{\mu+|\alpha|\leq N_0+\mu_0\atop \mu\leq\mu_0}\|L^{\mu}\partial^{\alpha}\partial u(t, \cdot)\|_{L^2(\partial\mo)}^2\\[2mm]
\lesssim& \sum\limits_{\mu+|\alpha|\leq N_0+\mu_0+1\atop \mu\leq \mu_0}\|L^{\mu}\partial^{\alpha}\partial u(t, \cdot)\|_{L^2(|x|\leq 1)}^2.
\end{aligned}
\end{equation*}
Substituting this into \eqref{rotationnl1} yields \eqref{rotationnl}.
\qquad\qquad \qquad\qquad \qquad \qquad \qquad \qquad \qquad \qquad  $\square$\end{proof}

\begin{lem}\label{Localest1} Suppose that $\Box u=0$ for $|x|>4$, $\partial_{\boldsymbol{\nu}}u=0$ on $\mathbb{R}_{+}\times\partial\mo$,
and $u\equiv 0$ for $t\leq 0$. Then
\begin{equation}\label{locale1}\begin{aligned}
&\sum\limits_{\mu+|\alpha|\leq N_0+\mu_0\atop \mu\leq \mu_0}\|L^{\mu}\partial^{\alpha}\partial u(t, \cdot)\|_{L^2(|x|\leq 2)}\\[2mm]
\lesssim&\sum\limits_{\mu+|\alpha|\leq N_0+\mu_0-1\atop \mu\leq \mu_0}\|L^{\mu}\partial^{\alpha}\Box u(t, \cdot)\|+\int_0^{t}e^{-\boldsymbol{c}/2 (t-s)}\sum\limits_{\mu+|\alpha|\leq N_0+\mu_0\atop \mu\leq \mu_0}\|L^{\mu}\partial^{\alpha}\Box u(s, \cdot)\|ds.
\end{aligned}
\end{equation}
\end{lem}

\begin{proof}
It follows from \eqref{spacenl4} and the definition of $L$  that
\begin{equation}\label{localest1}\begin{aligned}
&\sum\limits_{\mu+|\alpha\leq N_0+\mu_0\atop \mu\leq \mu_0}\|L^{\mu}\partial^{\alpha}\partial u(t, \cdot)\|_{L^2(|x|\leq 2)}\\[2mm]
\lesssim&\sum\limits_{\mu+\beta\leq N_0+\mu_0\atop \mu\leq\mu_0}t^{\mu}\|\partial_0^{\mu+\beta}\partial u(t, \cdot)\|_{L^{2}(|x|\leq 4)}+\sum\limits_{\mu+\beta\leq N_0+\mu_0-1\atop \mu\leq\mu_0}\|L^{\mu}\partial^{\alpha}\Box u(t, \cdot)\|_{L^2(|x|\leq 4)}.
\end{aligned}
\end{equation}
Next we manage to control the first term in the right hand side of \eqref{localest1}.
By $\partial_{\boldsymbol{\nu}}\partial_0^{\mu+\beta}u=0$ on $\Bbb R_+\times\partial\mo$, it follows from Lemma \ref{Exponetial}
that,
\begin{equation*}\begin{aligned}\label{localest2}
&\sum\limits_{\mu+\beta\leq N_0+\mu_0\atop \mu\leq \mu_0} t^{\mu}\|\partial_0^{\mu+\beta}\partial u(t, \cdot)\|_{L^2(|x|\leq 4)}\\[2mm]
\lesssim&\sum\limits_{\mu+\beta\leq N_0+\mu_0\atop \mu\leq \mu_0}t^{\mu}\int_0^{t}e^{-\c(t-s)}\|\partial_0^{\mu+\beta}\Box u(s, \cdot)\|ds\\[2mm]
\lesssim&\sum\limits_{\mu+\beta\leq N_0+\mu_0\atop \mu\leq \mu_0}\left\{\int_0^{t}e^{-\c(t-s)}(t-s)^{\mu}\|\partial_0^{\mu+\beta}\Box u(s, \cdot)\|ds+\int_0^{t}e^{-\c(t-s)}s^{\mu}\|\partial_0^{\mu+\beta}\Box u(s, \cdot)\|ds\right\}\\[2mm]
\lesssim&\sum\limits_{\mu+|\alpha|\leq N_0+\mu_0\atop \mu\leq\mu_0}\int_0^{t}e^{-\c/2(t-s)}\|L^{\mu} \partial^{\alpha}\Box u(s, \cdot)\|ds.
\end{aligned}
\end{equation*}
Combining this with \eqref{localest1} yields \eqref{locale1}.
\qquad\qquad \qquad\qquad \qquad \qquad \qquad \qquad \qquad \qquad  $\square$\end{proof}

\begin{lem}\label{Spacelocal} Suppose $\partial_{\boldsymbol{\nu}}u=0$ on $\Bbb R_+\times\partial\mo$ and $u\equiv 0$ when $t\leq 0$.
Then for fixed integers $\mu_0$ and $N_0$, we have
\begin{equation}\label{spacelo1}\begin{aligned}
&\sum\limits_{\mu+|\alpha|\leq N_0+\mu_0\atop \mu\leq \mu_0}\|L^{\mu}\partial^{\alpha}\partial u(t, \cdot)\|_{L^2(|x|\leq 2)}\\[2mm]
\lesssim &\sum\limits_{\mu+|\alpha|\leq N_0+\mu_0\atop \mu\leq \mu_0}\int_0^{t}e^{-\frac{\boldsymbol{c}}{2}(t-s)}\|L^{\mu}\partial^{\alpha}\Box u(s, \cdot)\|_{L^2(|x|\leq 4)} ds+\sum\limits_{\mu+|\alpha|\leq N_0+\mu_0-1\atop \mu\leq \mu_0}\|L^{\mu}\partial^{\alpha}\Box u(t, \cdot)\|_{L^2(|x|\leq 4)}\\[2mm]
&+\sum\limits_{\mu+|\alpha|\leq N_0+\mu_0+1\atop \mu\leq \mu_0}\int_0^{t}e^{-\frac{\boldsymbol{c}}{2}(t-s)}\left(\int_0^{s}\|L^{\mu}\partial^{\alpha}\Box u(\tau, \cdot)\|_{L^2(\{x: ||x|-(s-\tau)|<10\})}d\tau\right)ds\\[2mm]
&+\sum\limits_{\mu+|\alpha|\leq N_0+\mu_0\atop \mu\leq \mu_0}\int_0^{t}\|L^{\mu}\partial^{\alpha}
\Box u(s, \cdot)\|_{L^2(\{x:||x|-(t-s)|<10\})}ds.
\end{aligned}
\end{equation}
In addition, if $t>2$,
\begin{equation}\label{spacelo2}\begin{aligned}
&\sum\limits_{\mu+|\alpha|\leq N_0+\mu_0\atop \mu\leq\mu_0}\int_0^{t}\|L^{\mu}\partial^{\alpha} u(s, \cdot)\|_{L^2(|x|\leq 2)}ds\\
\lesssim&\sum\limits_{\mu+|\alpha|\leq N_0+\mu_0+1\atop\mu\leq\mu_0}\int_0^{t}\int_0^{s}\|L^{\mu}\partial^{\alpha}\Box u(\tau, \cdot)\|_{L^2(\{x:||x|-(s-\tau)|<10\})}d\tau ds.
\end{aligned}
\end{equation}
\end{lem}
\begin{proof} It suffices only to prove \eqref{spacelo1} since \eqref{spacelo2} directly comes from \eqref{spacelo1}.

Assume that $u$ satisfies
\begin{equation*}\label{L0}\begin{cases}
\Box u=G(t,x), \ \ (t, x)\in \mathbb{R}_{+}\times\mo,\\[2mm]
\partial_{\boldsymbol{\nu}}u=0,\ \ (t, x)\in\mathbb{R}_{+}\times\partial\mo,\\[2mm]
u(t, x)\equiv 0,\ \ t\leq 0.
\end{cases}
\end{equation*}
Let $u_i\ (i=1, 2)$ satisfy
\begin{equation*}\label{L1}\begin{cases}
\Box u_i=(2-i+(2i-3)\varrho(x/2))G(t,x), \ \ (t, x)\in \mathbb{R}_{+}\times\mo,\\[2mm]
\partial_{\boldsymbol{\nu}}u_i=0,\ \ (t, x)\in\mathbb{R}_{+}\times\partial\mo,\\[2mm]
u_i(t, x)\equiv 0,\ \ t\leq 0.
\end{cases}
\end{equation*}
Then $u=u_1+u_2$. It follows from Lemma \ref{Localest1} that
\eqref{spacelo1} holds for $u_1$. Next we estimate $u_2$. For this purpose,
we set $u_2=w_0+w_1$ where $w_0$ solves the following problem
\begin{equation}\label{spacelo4}\begin{cases}
\Box w_0=\Box u_2,\ \ |x|\geq 2;\qquad  0,\ \ \text{otherwise},\\[2mm]
w_0(t, x)\equiv 0,\ \ t\leq 0.
\end{cases}
\end{equation}
In addition, set $\hat u=(1-\varrho(x/2))w_0+w_1$. Then $\Box\hat u$ is supported in $\{|x|\leq 4\}$ due to
\begin{equation*}
\Box\hat u=\nabla\varrho(x/2)\cdot\nabla w_0+2\Delta\varrho(x/2)w_0.
\end{equation*}
Consequently, as in the treatment on $u_1$, by Lemma \ref{Localest1} and the fact that $u_2=\hat u$ when $|x|\leq 2$, one has that
\begin{equation}\label{spacelo3}\begin{aligned}
&\sum\limits_{\mu+|\alpha|\leq N_0+\mu_0\atop \mu\leq\mu_0}\|L^{\mu}\partial^{\alpha}\partial u_2(t, \cdot)\|_{L^2(|x\leq 2)}\\[2mm]
=&\sum\limits_{\mu+|\alpha|\leq N_0+\mu_0\atop \mu\leq\mu_0}\|L^{\mu}\partial^{\alpha}\partial \hat u(t, \cdot)\|_{L^2(|x\leq 2)}\\[2mm]
\lesssim&\sum\limits_{\mu+|\alpha|\leq N_0+\mu_0\atop\mu\leq\mu_0}\int_0^{t}e^{-\frac{\boldsymbol{c}}{2}(t-s)}\|L^{\mu}\partial^{\alpha}\Box\hat u(s, \cdot)\|ds+\sum\limits_{\mu+|\alpha|\leq N_0+\mu_0-1\atop\mu\leq \mu_0}\|L^{\mu}\partial^{\alpha}\Box\hat u(t, \cdot)\|\\[2mm]
\lesssim& \sum\limits_{\mu+|\alpha|\leq N_0+\mu_0+1\atop\mu\leq\mu_0}\int_0^{t}e^{-\frac{\boldsymbol{c}}{2}(t-s)}\|L^{\mu}\partial^{\alpha}w_0(s, \cdot)\|_{L^2(|x|\leq 4)}ds+\sum\limits_{\mu+|\alpha|\leq N_0+\mu_0\atop\mu\leq\mu_0}\|L^{\mu}\partial^{\alpha}w_0(t, \cdot)\|_{L^2(|x|\leq 4)}.
\end{aligned}
\end{equation}
Since $w_0$ satisfies the boundaryless problem \eqref{spacelo4}, one has that for any $\mu$ and $\alpha$,
\begin{equation*}\begin{aligned}
\|L^{\mu}\partial^{\alpha}w_0(t, \cdot)\|_{L^2(|x|\leq 4)}&\lesssim \sum\limits_{\nu+|\beta|\leq \mu+|\alpha|;\ \nu\leq \mu}\int_0^{t}\|L^{\nu}\partial^{\beta}\Box w_0(\tau, \cdot)\|_{L^2(\{x: ||x|-(t-\tau)|<10\})}d\tau\\[2mm]
&\lesssim \sum\limits_{\nu+|\beta|\leq \mu+|\alpha|;\ \nu\leq \mu}\int_0^{t}\|L^{\nu}\partial^{\beta}\Box u(\tau, \cdot)\|_{L^2(\{x: ||x|-(t-\tau)|<10\})}d\tau.
\end{aligned}
\end{equation*}
Combining this with \eqref{spacelo3} and the estimate of $u_1$ yields \eqref{spacelo1}. And then we complete the proof of Lemma \ref{Spacelocal}.
\qquad\qquad \qquad\qquad \qquad \qquad \qquad \qquad \qquad \qquad \qquad
\qquad \qquad \qquad \qquad \qquad  $\square$\end{proof}

\subsection{Time-space energy estimates}

In this subsection, some key weighted $L^2$ time-space energy estimates, which can be thought as the modified
version of Keel-Smith-Sogge estimate in \cite{KSS4}, are established for the solution
$u$ of \eqref{VLNV}. Denote by $S_T=[0, T]\times\mo$ for any $T>0$.

\begin{prop}\label{MKSS} Suppose that the smooth function $u$ admits $\partial_{\boldsymbol{\nu}}u=0$ on $\Bbb R_+\times\partial\mo$
and $u\equiv 0$ when $t\leq 0$. Then for fix integers $N_0$ and $\mu_0$, we have
\begin{equation}\label{KKspace}\begin{aligned}
&(\ln(2+T))^{-1/2}\sum\limits_{\mu+|\alpha|\leq N_0+\mu_0\atop \mu\leq \mu_0}\|\left<x\right>^{-1/2}L^{\mu}\partial^{\alpha}\partial u\|_{L^2(S_T)}\\[2mm]
\lesssim& \int_0^{T}\sum\limits_{\mu+|\alpha|\leq N_0+\mu_0+1\atop \mu\leq \mu_0}\|\Box L^{\mu}\partial^{\alpha} u(s, \cdot)\| ds+\sum\limits_{\mu+|\alpha|\leq N_0+\mu_0\atop \mu\leq \mu_0}\|\Box L^{\mu}\partial^{\alpha} u\|_{L^2(S_T)}.
\end{aligned}
\end{equation}
 Additionally,
 \begin{equation}\label{KKrotation}\begin{aligned}
&(\ln(2+T))^{-1/2}\sum\limits_{\mu+|\alpha|\leq N_0+\mu_0\atop \mu\leq \mu_0}\|\left<x\right>^{-1/2}L^{\mu} Z^{\alpha}\partial u\|_{L^2(S_T)}\\[2mm]
\lesssim& \int_0^{T}\sum\limits_{\mu+|\alpha|\leq N_0+\mu_0+1\atop \mu\leq \mu_0}\|\Box L^{\mu} Z^{\alpha} u(s, \cdot)\| ds+\sum\limits_{\mu+|\alpha|\leq N_0+\mu_0\atop \mu\leq \mu_0}\|\Box L^{\mu} Z^{\alpha} u\|_{L^2(S_T)}.
\end{aligned}
\end{equation}

\end{prop}

In order to prove Proposition \ref{MKSS}, we shall establish the next two lemmas in advance. The first one can be looked as a simple version of Proposition \ref{MKSS} and the second one is the local-in-space $L^2-$time-space energy estimates.

\begin{lem}\label{Decaynl} If the smooth function $u$ admits $\partial_{\boldsymbol{\nu}}u=0$ on $\Bbb R_+\times\partial\mo$
and $u\equiv 0$ when $t\leq 0$, then for any nonnegative integer $\beta$,
\begin{equation}\label{decaynl5}
(\ln(2+T))^{-1/2}\|\left<x\right>^{-1/2}\partial\partial_0^{\beta}u\|_{L^2(S_T)}+\|\partial\partial_0^{\beta}u\|_{L^2([0, T]\times \{|x|\leq 4\})}\lesssim \int_0^{T}\|\Box\partial_0^{\beta}u(s, \cdot)\|ds.
\end{equation}
\end{lem}

\begin{proof}
We only need to prove \eqref{decaynl5} for $\beta=0$ since other cases can
be proved in the same way due to $\partial_{\boldsymbol{\nu}}\partial_0^{i}u=0$ on $\mathbb{R}_{+}\times\partial\mo$ for $0\leq i\leq \beta$.
It follows from the energy estimate in Lemma \ref{Timenl} that
\begin{equation*}
\|\partial u(t, \cdot)\|\leq 2\int_0^{t}\|\Box u(s, \cdot)\| ds.
\end{equation*}
This immediately implies
\begin{equation}\label{decaynl2}
(1+T)^{-1/2}\|\partial u\|_{L^2(S_T)}\lesssim \int_0^{T}\|\Box u(s, \cdot)\|ds.
\end{equation}

We continue to use the notation of $u_i\ (i=1, 2)$ defined in the proof of Lemma \ref{Spacelocal}. With respect to $u_1$, by using \eqref{exponential} in Lemma \ref{Exponetial}, one has
\begin{equation}\label{decaynl4}\begin{aligned}
\|\partial u_1\|_{L^2([0, T]\times\{|x|\leq 4\})}&\lesssim \|\int_0^{t}e^{-\c(t-s)}\|\Box u(s, \cdot)\|ds\|_{L^2([0, T])}\\[2mm]
&\lesssim \int_0^{T}\|\Box u(s, \cdot)\|ds,
\end{aligned}
\end{equation}
where the last inequality comes from the Young inequality.

When it comes to $u_2$, based on the estimate \eqref{spacelo3} in Lemma \ref{Spacelocal} and \eqref{weighted1} in Lemma \ref{Weighted}, one has
\begin{equation*}
\|\partial u_2\|_{L^2([0, T]\times\{|x|\leq 4\})}\lesssim \int_0^{T}\|\Box u(s, \cdot)\|ds.
\end{equation*}
Combining this with \eqref{decaynl4} yields
\begin{equation}\label{decaynl3}
\|\partial u\|_{L^2([0, T]\times \{|x|\leq 4\})}\lesssim \int_0^{T}\|\Box u(s, \cdot)\|ds.
\end{equation}
Based on  \eqref{decaynl2} and \eqref{decaynl3}, by the scaling method as in the proof of Lemma 2.3 in \cite{KSS4},
we can get \eqref{decaynl5} for $\beta=0$ and then the proof of Lemma \ref{Decaynl} is completed.
\end{proof}

\begin{lem}\label{Localest2} Under the assumptions in Proposition \ref{MKSS}, we have
\begin{equation}\label{localest3}\begin{aligned}
&\sum\limits_{\mu+|\alpha|\leq N_0+\mu_0\atop \mu\leq \mu_0}\|L^{\mu}\partial^{\alpha}\partial u\|_{L^2([0, T]\times\{|x|\leq 2\})}\\[2mm]
\lesssim&\sum\limits_{\mu+|\alpha|\leq N_0+\mu_0-1\atop \mu\leq\mu_0}\|\Box L^{\mu}\partial^{\alpha}u\|_{L^2(S_T)}+\int_0^{T}\sum\limits_{\mu+|\alpha|\leq N_0+\mu_0\atop \mu\leq \mu_0}\|\Box L^{\mu}\partial^{\alpha} u(s, \cdot)\| ds.
\end{aligned}
\end{equation}
\end{lem}

\begin{proof} We use some notations in the proof of Lemma \ref{Spacelocal}. The estimate of $u_1$ in \eqref{localest3} comes from \eqref{locale1} by integrating the variable $t$ from $0$ to $T$ on both sides of \eqref{locale1} and by applying Young inequality.

For $u_2$, one has
\begin{equation}\label{3.1}\begin{aligned}
&\sum\limits_{\mu+|\alpha|\leq N_0+\mu_0\atop \mu\leq\mu_0}\|L^{\mu}\partial^{\alpha}\partial u\|_{L^2([0, T]\times (|x|\leq 2))}\\[2mm]
=&\sum\limits_{\mu+|\alpha|\leq N_0+\mu_0\atop \mu\leq\mu_0}\|L^{\mu}\partial^{\alpha}\partial \hat u\|_{L^2([0, T]\times (|x|\leq 2))}\\[2mm]
\lesssim&\sum\limits_{\mu+|\alpha|\leq N_0+\mu_0\atop \mu\leq\mu_0}\|\Box L^{\mu}\partial^{\alpha}\hat u\|_{L^2(S_T)}\\[2mm]
\lesssim&\sum\limits_{\mu+|\alpha|\leq N_0+\mu_0+1\atop \mu\leq\mu_0}\|L^{\mu}\partial^{\alpha} w_0\|_{L^2([0, T]\times \{|x|\leq 4\})}.
\end{aligned}
\end{equation}
Here the first inequality in \eqref{3.1} comes from the application  of Lemma \ref{Localest1} and Young's inequality.
Thus, it follows from \eqref{3.1} and estimate \eqref{weighted1} in Lemma \ref{Weighted} to $w_0$ that \eqref{localest3}
is proved.\quad  $\square$\end{proof}

We now start the proof of Proposition \ref{MKSS}.

\vskip 0.2 true cm

{\bf Proof of Proposition \ref{MKSS}:} Similar to \eqref{decaynl2} and \eqref{decaynl3} in Lemma \ref{Decaynl}, we have 
that for $0\leq j\leq N_0$,
\begin{equation*}\label{mkss1}
(1+t)^{-1/2}\|\partial\partial_0^j u\|_{L^2(S_t)}\lesssim \int_0^{t}\|\Box \partial_0^j u(s, \cdot)\|_2 ds
\end{equation*}
and
\begin{equation*}\label{mkss2}
\|\partial\partial_0^j u\|_{L^2([0, t]\times\{|x|\leq 4\})}\lesssim\int_0^{t}\|\Box\partial_0^j u(s, \cdot)\|_2 ds.
\end{equation*}
In addition, due to the elliptic estimate in Lemma \ref{Elliptic}, we have
\begin{equation}\label{mkss3}\begin{aligned}
&(1+t)^{-1/2}\sum\limits_{|\alpha|\leq N_0}\|\partial^{\alpha}\partial u\|_{L^2(S_t)}\\[2mm]
\lesssim& (1+t)^{-1/2}\sum\limits_{|\alpha|\leq N_0-1}\|\Box\partial u\|_{L^2(S_t)}
+\sum\limits_{|\alpha|\leq N_0}\int_0^{t}\|\Box\partial^{\alpha} u(s, \cdot)\|_2 ds
\end{aligned}
\end{equation}
and
\begin{equation}\label{mkss4}\begin{aligned}
&\sum\limits_{|\alpha|\leq N_0}\|\partial^{\alpha}\partial u\|_{L^2([0, t]\times\{|x|\leq 2\})}\\[2mm]
\lesssim&\sum\limits_{|\alpha|\leq N_0-1}\|\Box\partial^{\alpha} u\|_{L^2([0, t]\times\{|x|\leq 4\})}+\sum\limits_{|\alpha|\leq N_0}\int_0^{t}\|\Box\partial^{\alpha} u(s, \cdot)\|_2 ds.
\end{aligned}
\end{equation}
Together with the scaling techniques in \cite{KSS4}, it follows from \eqref{mkss3} and \eqref{mkss4} that
\begin{equation}\label{mkss5}\begin{aligned}
&(\ln(2+T))^{-1/2}\sum\limits_{|\alpha|\leq N_0}\|\left<x\right>^{-1/2}\partial^{\alpha}\partial u\|_{L^2(S_{T})}\\[2mm]
\lesssim&\sum\limits_{|\alpha|\leq N_0-1}\|\Box\partial^{\alpha} u\|_{L^2(S_T)}+\sum\limits_{|\alpha|\leq N_0}\int_0^{T}\|\Box\partial^{\alpha}u(s, \cdot)\|ds.
\end{aligned}
\end{equation}
On the other hand,  it follows from Sobolev trace theorem and Lemma \ref{Rotationnl} with $h^{\beta\gamma}=0$ that for fixed $\alpha\in\Bbb N_0$,
\begin{equation}\label{mkss6}\begin{aligned}
\|\partial L^{\mu_0} Z^{\alpha}u(t, \cdot)\|^2
\lesssim&\int_0^{t}\|\Box L^{\mu_0} Z^{\alpha} u(s, \cdot)\|\cdot\|\partial L^{\mu_0} Z^{\alpha} u(s, \cdot)\| ds\\[2mm]
&+\sum\limits_{|\beta|\leq 1}\int_0^{t}\|L^{\mu_0}\partial^{\alpha+\beta}\partial u(s, \cdot)\|_{L^2(|x|\leq 2)}^2 ds.
\end{aligned}
\end{equation}
For any fixed $T>0$, there exists a maximum $t_0\in [0, T]$ such that for $0\leq t\leq t_0$,
\begin{equation*}\int_0^{t}\|\Box L^{\mu_0}Z^{\alpha} u(s, \cdot)\|\cdot\|\partial L^{\mu_0} Z^{\alpha}u(s, \cdot)\|ds\leq\sum_{|\beta|\leq 1}\int_0^{T}\|\partial^{\beta}\partial L^{\mu_0} Z^{\alpha} u(s, \cdot)\|_{L^2(|x|\leq 2)}^2 ds.
\end{equation*}
Combining this with \eqref{mkss6} yields that for $0\leq t\leq t_0$,
\begin{equation}\label{mkss7}
\|\partial L^{\mu_0} Z^{\alpha} u(t, \cdot)\|^2
\lesssim\sum_{|\beta|\leq 1}\int_0^{T}\|\partial^{\beta}\partial L^{\mu_0} Z^{\alpha} u(s, \cdot)\|_{L^2(|x|\leq 2)}^2 ds,
\end{equation}
and for  $t_0\leq t\leq T$,
\begin{equation}\label{mkss8}
\|\partial L^{\mu_0} Z^{\alpha} u(t, \cdot)\|^2\lesssim \int_0^{t}\|\Box L^{\mu_0}Z^{\alpha} u(s, \cdot)\|\cdot\|\partial L^{\mu_0} Z^{\alpha}u(s, \cdot)\|ds.
\end{equation}
If there exists a $t^*\in [0, t_0]$ such that
\begin{equation*}
\|\partial L^{\mu_0}Z^{\alpha}u(t^*, \cdot)\|=\sup\limits_{0<t<T}\|\partial L^{\mu_0}Z^{\alpha}u(t, \cdot)\|,
\end{equation*}
then it follows from \eqref{mkss7} that for $0\leq t\leq T$,
\begin{equation}\label{mkss9}\begin{aligned}
\|\partial L^{\mu_0} Z^{\alpha} u(t, \cdot)\|^2
\lesssim\sum_{|\beta|\leq 1}\int_0^{T}\|\partial^{\beta}
\partial L^{\mu_0} Z^{\alpha} u(s, \cdot)\|_{L^2(|x|\leq 2)}^2 ds.
\end{aligned}
\end{equation}
If there exists a $t_{*}\in [t_0, T]$ such that
\begin{equation*}
\|\partial L^{\mu_0} Z^{\alpha} u(t_{*}, \cdot)\|=\sup\limits_{0<t<T}\|\partial L^{\mu_0} Z^{\alpha} u(t, \cdot)\|,
\end{equation*}
then it follows from \eqref{mkss8} and the definition of $t_{*}$ that for $0\leq t\leq T$,
\begin{equation}\label{mkss10}\begin{aligned}
&\|\partial L^{\mu_0} Z^{\alpha}u(t, \cdot)\|
\leq \|\partial L^{\mu_0}Z^{\alpha}u(t_{*}, \cdot)\|\\[2mm]
\lesssim&\int_0^{t_{*}}\|\Box L^{\mu_0} Z^{\alpha} u(s, \cdot)\| ds
\leq \int_0^{T}\|\Box L^{\mu_0} Z^{\alpha} u(s, \cdot)\| ds.
\end{aligned}
\end{equation}
Finally, it follows from \eqref{mkss9} and \eqref{mkss10} that
\begin{equation*}\label{mkss11}\begin{aligned}
\|\partial L^{\mu_0} Z^{\alpha} u(t, \cdot)\|^2
&\lesssim \left(\int_0^{T}\|\Box L^{\mu_0} Z^{\alpha}u(s, \cdot)\| ds\right)^2
+\sum\limits_{\mu\leq \mu_0\atop |\beta|\leq |\alpha|+1}\int_0^{T}\|L^{\mu}\partial^{\beta}\partial u(s, \cdot)\|_{L^2(|x|\leq 2)}^2 ds.
\end{aligned}
\end{equation*}
Simultaneously, one has that
\begin{equation*}\label{mkss12}\begin{aligned}
&(1+T)^{-1/2}\|\partial L^{\mu_0} Z^{\alpha} u\|_{L^2(S_T)}\\[2mm]
\lesssim&\sum\limits_{\mu\leq \mu_0\atop |\beta|\leq |\alpha|+1}\|L^{\mu} \partial^{\beta}\partial u(s, \cdot)\|_{L^2([0, T]\times \{|x|\leq 2\})}+\int_0^{T}\|\Box L^{\mu_0} Z^{\alpha} u(s, \cdot)\| ds.
\end{aligned}
\end{equation*}
Combining this with \eqref{mkss5}, Lemma \ref{Localest2} and the scaling method in \cite{KSS4} yields Proposition \ref{MKSS}.$\square$

\section{Pointwise decay estimate}\label{IV}
In this section, we derive the pointwise decay estimate of the smooth solution $w$ to the following problem
\begin{equation}\label{LNV}\begin{cases}
\Box w=F(t, x),\qquad  \ (t, x)\in \mathbb{R}_{+}\times\mo,\\[2mm]
\partial_{\boldsymbol{\nu}}w=0, \qquad\qquad (t, x)\in\mathbb{R}_{+}\times\partial\mo,\\[2mm]
w(t, x)\equiv 0,\qquad\qquad t\leq 0.
\end{cases}
\end{equation}

\begin{thm}\label{Point} Suppose that $w$ is a smooth solution of problem \eqref{LNV}. Then for integers $\mu $ and $\alpha$, one has
that
\begin{equation}\label{point1}\begin{aligned}
&(1+t+|x|)|L^{\mu} Z^{\alpha} w(t, x)|\\[2mm]
\lesssim &\int_0^{t}\int_{\mo}\sum\limits_{\nu+|\beta|\leq |\alpha|+\mu+5\atop\nu\leq \mu+1}|L^{\nu} Z^{\beta} F(s, y)|\frac{dyds}{|y|}\\[2mm]
&+\int_0^{t}\sum\limits_{\nu+|\beta|\leq |\alpha|+\mu+2\atop\nu\leq \mu+1}\|L^{\nu}\partial^{\beta}F(s, \cdot)\|_{L^2(|x|\leq 2)}ds.
\end{aligned}
\end{equation}

\end{thm}

\begin{proof} At first, we consider the following inhomogeneous wave equation
\begin{equation}\label{free}\begin{cases}
\Box w_0=G_0,\ \ (t, x)\in\mathbb{R}_{+}\times\mathbb{R}^3,\\[2mm]
w_0(t, x)\equiv 0,\ \ t\leq 0.
\end{cases}
\end{equation}
It follows from Proposition 2.1 and Lemma 2.2 in \cite{KSS3} (one can also see Theorem 3.1 in \cite{MC2}) that
\begin{equation}\label{free1}
(1+t+|x|)|w_0(t,x)|\lesssim\int_0^{t}\int_{\mathbb{R}^3}\sum\limits_{\nu+|\beta|\leq 3\atop \nu\leq 1}|L^{\nu} Z^{\beta} G_0(s, y)|\frac{dyds}{|y|}
\end{equation}
and
\begin{equation}\label{free2}
|x||w_0(t,x)|\lesssim\int_0^{t}\int_{||x|-(t-s)|}^{|x|+(t-s)}\sup\limits_{\omega\in\mathbb{S}^2}|G_0(s, r\omega)|r dr ds.
\end{equation}
Set $w_0(t, x)=\varrho(x)L^{\mu} Z^{\alpha}w(t, x)$. Then
\begin{equation*}\begin{cases}
\Box w_0=\varrho(x)\Box  L^{\mu}Z^{\alpha} w+\hat G_0(t, x),\\[2mm]
w_0(t, x)\equiv 0, \ t\leq 0,
\end{cases}
\end{equation*}
where $\hat G_0(t, x)=-2\nabla\varrho(x)\cdot\nabla L^{\mu}Z^{\alpha}w-(\Delta\varrho)(x)L^{\mu}Z^{\alpha}w.$
We decompose $w_0$ as $w_0=\hat w_0+(w_0-\hat w_0)$ with  $\hat w_0$ satisfying the following inhomogeneous equation
\begin{equation*}\begin{cases}
\Box \hat w_0=\varrho(x)\Box L^{\mu} Z^{\alpha} w=\varrho(x)\sum\limits_{\nu\leq \mu}C_{\nu\mu}L^{\nu}Z^{\alpha} F,\\[2mm]
\hat w_0(t, x)\equiv 0,\qquad  \ t\leq 0.
\end{cases}
\end{equation*}
By \eqref{free} and \eqref{free1}, we have that
\begin{equation}\label{free3}
(1+t+|x|)|\hat w_0(t,x)|
\lesssim\int_0^{t}\int_{\mo}\sum\limits_{\nu+|\beta|\leq \mu+|\alpha|+3\atop\nu\leq \mu+1}| L^{\nu}Z^{\beta}F(s, y)|\frac{dyds}{|y|}.
\end{equation}
In addition, due to the definition of $\hat w_0$, it also follows from \eqref{free2} that
\begin{equation}\label{free4}\begin{aligned}
|x||(w_0-\hat w_0)(t, x)|\lesssim&\int_0^{t}\int_{||x|-(t-s)|}^{|x|+(t-s)}\sup\limits_{\omega\in \mathbb{S}^2}|\hat G_0(s, r\omega)|r dr ds.\\[2mm]
\end{aligned}
\end{equation}
Note that $\hat G_0(t, y)\equiv 0$ for $|y|\leq 1$ and $|y|\geq 2$. Then the time integrand in \eqref{free4} is not zero only for
$-2\leq |x|-(t-s)\leq 2$ (or equivalently, $(t-|x|)-2\leq s\leq (t-|x|)+2$).
This, together with \eqref{free4} and the definition of $\hat G_0$, yields
\begin{equation*}
|x||(w_0-\hat w_0)(t, x)|\lesssim\frac{1}{1+|t-|x||}\sup\limits_{(t-|x|)-2\leq s\leq (t-|x|)+2\atop |y|\leq 2}(1+s)\sum\limits_{|\beta|\leq 1}|\partial^{\beta} L^{\mu}Z^{\alpha} w|(s, y).
\end{equation*}
Combining this with \eqref{free3} shows
\begin{equation}\label{free5}\begin{aligned}
(1+t+|x|)|L^{\mu}Z^{\alpha} w(t, x)|&\lesssim \int_0^{t}\int_{\mo}\sum\limits_{\nu+|\beta|\leq\mu+|\alpha|+3\atop \nu\leq \mu+1}|L^{\nu}Z^{\beta} F(s, y)|\frac{dyds}{|y|}\\[2mm]
&\quad +\sup\limits_{0\leq s\leq t\atop |y|\leq 2}(1+s)\sum\limits_{\nu+|\beta|\leq \mu+|\alpha|+1\atop \nu\leq \mu}|L^{\nu}\partial^{\beta}w|(s, y).
\end{aligned}
\end{equation}
Next we manage to control the last term in the right hand side of \eqref{free5}. To this end, it suffices to estimate the
term $\sup\limits_{|x|\leq 2}t|w(t, x)|$ since the other cases can be treated in the similar way with the help of Lemma \ref{Conclusion} and Lemma \ref{Exponetial}. To this end, with the decomposition of $w$ similar to that of $u$ in the proof of Lemma \ref{Spacelocal}, we only need to consider two cases $F(t, x)=0$ for $|x|\geq 4$ and $F(t, x)=0$ for $|x|\leq 2$.

When $F(t, x)=0$ for $|x|\geq 4$, one has,
\begin{equation*}
t|w(t, x)|\leq\sum\limits_{j=0, 1}\int_0^{t}|(s\partial_0)^{j}w(s, x)|ds.
\end{equation*}
This, together with Sobolev imbedding lemma, \eqref{exponential} in Lemma \ref{Exponetial}  and Remark \ref{Ipdecay}, yields
\begin{equation}\label{free6}\begin{aligned}
t\sup_{|x|\leq 2}|w(t, x)|&\lesssim \sum\limits_{0\leq j\leq 1;\ |\sigma|\leq 2}\int_0^{t}s^{j}\|\partial_0^{j}\nabla^{\sigma}w(s, x)\|_{L^2(|x|\leq 2)}ds\\[2mm]
&\lesssim \sum\limits_{l\leq 1;\ j\leq 1}\int_0^{t}s^{l}\int_0^{s}e^{-\c(s-\tau)}\|\partial_0^{j+l}F(\tau, \cdot)\|d\tau ds+\sum\limits_{l\leq 1}\int_0^{t}s^{l}\|\partial_0^{l}F(s, \cdot)\|ds\\[2mm]
&\lesssim\sum\limits_{l\leq 1;\ j\leq 1}\int_0^{t}\int_0^{s}e^{-\c(s-\tau)}(s-\tau)^{l}\|\partial_0^{j+l}F(\tau, \cdot)\|d\tau ds+\sum\limits_{l\leq 1}\int_0^{t}s^{l}\|\partial_0^{l}F(s, \cdot)\|ds\\[2mm]
&\quad+\sum\limits_{l\leq 1;\ j\leq 1}\int_0^{t}\int_0^{s}e^{-\c(t-s)}\tau^{l}\|\partial_0^{j+l}F(\tau, \cdot)\|d\tau ds\\[2mm]
\lesssim&\sum\limits_{\mu+|\alpha|\leq 2;\ \mu\leq 1}\int_0^{t}\|L^{\mu}Z^{\alpha}F(s, \cdot)\|ds,
\end{aligned}
\end{equation}
where the last inequality comes from the Young inequality and the definition of $L$ and $Z$. Combining \eqref{free6} with \eqref{free5} yields \eqref{point1} when $F(t, x)=0$ if $|x|\geq 4$.

When $F(t, x)=0$ if $|x|\leq 2$, as in the proof of Lemma \ref{Spacelocal}, we set $w=w_0+w_r$ with $w_0$ defined by \eqref{free} and
the corresponding $G_0$ replaced by $F$. In addition, we set $\hat w_0=(1-\varrho(x/2))w_0+w_r$ with $\Box \hat w_0$ supported in $(|x|\leq 4)$
due to $\Box \hat w=\nabla\varrho(x/2)\cdot \nabla w_0+2\Delta \varrho(x/2)w_0.$
Then $w=\hat w$ when $|x|\leq 2$. Following the proof for the case of $F(t, x)=0$ when $|x|\geq 4$, we arrive at
\begin{equation}\label{free7}\begin{aligned}
&t\sup\limits_{|x|\leq 2}|L^{\nu}\partial^{\beta}w(t, x)|\\[2mm]
=&t\sup\limits_{|x|\leq 2}|L^{\nu}\partial^{\beta}\hat w(t, x)|\\[2mm]
\lesssim& \int_0^{t}\sum\limits_{l+|\sigma|\leq \nu+|\beta|+2\atop l\leq \nu+1}\|L^{l}\partial^{\sigma}\Box \hat w(s, \cdot)\|ds\\[2mm]
\lesssim&\int_0^{t}\sum\limits_{l+|\sigma|\leq \nu+|\beta|+3\atop l\leq \nu+1}\|L^{l}\partial^{\sigma}w_0(s, \cdot)\|_{L^{\infty}(2\leq |x|\leq 4)}ds.
\end{aligned}
\end{equation}
On the other hand, by \eqref{free2} and Lemma \ref{Sobolev},
\begin{equation}\label{3.2}\begin{aligned}
&\ds \sup_{2\leq |x|\leq 4}|L^{l}\partial^{\sigma}w_0(s, x)|\\[2mm]
\lesssim&\sum\limits_{l'+|\sigma'|\leq l+|\sigma|\atop l'\leq l}\int_0^{s}\int_{||x|-(s-\tau)|}^{|x|+(s-\tau)}\sup\limits_{\omega\in\mathbb{S}^2}|L^{l'}Z^{\sigma'}F(\tau, r\omega)|rdrd\tau\\[2mm]
\lesssim& \sum_{l'+|\sigma'|\leq l+|\sigma|+2\atop l'\leq l}\int_0^{s}\int_{|s-\tau-|y||\leq 4}|L^{l'} Z^{\sigma'}F(\tau, y)|\frac{dyd\tau}{|y|}.
\end{aligned}
\end{equation}
Note that $\{(\tau, y): 0\leq \tau\leq s, |s-\tau-|y||\leq 4\}\cap\{(\tau, y): 0\leq \tau\leq s', |s'-\tau-|y||\leq 4\}=\varnothing$
when $|s-s'|\geq 10$. Then  substituting \eqref{3.2} into \eqref{free7} leads to
\begin{equation*}
\sup\limits_{|y|\leq 2}(1+t)\sum\limits_{\nu+|\beta|\leq \mu+|\alpha|+1\atop \nu\leq \mu}|L^{\nu}\partial^{\beta}w|(s, y)\lesssim\int_0^{t}\int_{\mo}\sum\limits_{\nu+|\beta|\leq \mu+|\alpha|+5\atop \nu\leq \mu+1}|L^{\nu}Z^{\beta}F(s, y)|\frac{dyds}{y}.
\end{equation*}
Combining this with \eqref{free5} and \eqref{free6} shows \eqref{point1}.\qquad
\qquad \qquad \qquad \qquad \qquad \qquad \qquad \qquad $\square$\end{proof}

\section{Continuity induction argument and proof of Theorem \ref{them1}}\label{V}

In this section, based on Sect. \ref{II}-Sect. \ref{IV} and the local existence of problem \eqref{Mpro}
shown in Sect. \ref{VI} below,
we start to prove Theorem \ref{them1} by the continuity induction argument.

It follows from Theorem \ref{Lexistence} and Remark \ref{Rweighted} in Sect. \ref{VI} that
the solution $u$ of problem \eqref{Mpro} uniquely exists for $0\leq t\leq 4$ and
\begin{equation}\label{Start}
\sup\limits_{0<t<4}\sum\limits_{|\alpha|\leq 69}\|\left<x\right>^{\alpha}\partial^{\alpha}u(t, \cdot)\|\lesssim\varepsilon.
\end{equation}
In addition, as shown in Sect. 5 of \cite{MC2}, the solution of problem \eqref{Mpro} exists in the
region $\{(t, x): 0\leq 2t\leq |x|\}$\footnote{If the initial date $(u_0, u_1)$ in \eqref{intialcon} is compactly supported, then due to the weak Huygens principle, the solution $u$ in this domain is identically zero.} and satisfies
\begin{equation}\label{5.0}
\sup\limits_{t\geq 0}\sum\limits_{|\alpha|\leq 69}\|\left<x\right>^{\alpha}\partial^{\alpha}u(t, \cdot)\|_{L^2(\{x: |x|\geq 2t\})}\lesssim \varepsilon.
\end{equation}
Combining \eqref{5.0} with \eqref{Start} yields that
\begin{equation}\label{Start1}
\sup\limits_{0<t<+\infty}\sum\limits_{|\alpha|\leq 69}\|\left<x\right>^{|\alpha|}\partial^{\alpha}u(t, \cdot)\|_{L^2(x\in\mo: (t, x)\in D_0)}\lesssim \varepsilon,
\end{equation}
where $D_0=\{(t, x): 0\leq 2t\leq |x|, x\in\mo\}\cap \{(t, x): 0\leq t\leq 4, x\in\mo\}$.

To prove the global existence theorem, we will do some preparations based on \eqref{Start1} first. Set
\begin{equation}\label{Start2}
u_0(t, x)=\zeta(t, x)u(t, x),
\end{equation}
where $\zeta(t, x)=\eta(t V(x))$. Choose the functions $\eta\in C^{\infty}(\mathbb{R})$ and $V(x)\in C^{\infty}(\mo)$ such that
\begin{equation*}\begin{aligned}
&\eta(s)=1\ \ \text{for $s\leq 1/8$};\ \ \eta(s)=0\ \ \text{for $s\geq 1/4$},\\[2mm]
&0<V(x)\leq 1;\ \ V(x)=1\ \ \text{for $|x|\leq 4$};\ \ V(x)=|x|^{-1}\ \ \text
{for $|x|\geq 8$}.
\end{aligned}
\end{equation*}
Direct computation yields
\begin{equation*}
\Box u_0=\zeta\mathcal{N}(\partial u, \partial^2 u)+[\Box, \zeta]u.
\end{equation*}
It follows from \eqref{Start1}-\eqref{Start2} and Lemma \ref{Sobolev} that there exists an absolute constant $\kappa_0$ such that
\begin{equation}\label{reduce2}
(1+t+|x|)\sum\limits_{\mu+|\alpha|\leq 67}|L^{\mu}Z^{\alpha}u_0(t,x)|+\sum\limits_{\mu+|\alpha|+|\beta|\leq 69}\|\left<t+r\right>^{|\beta|}L^{\mu}Z^{\alpha}\partial^{\beta}u_0(t, \cdot)\|\leq \kappa_0\varepsilon.
\end{equation}
Let $u=u_0+w$ with $w$ satisfying
\begin{equation}\label{reduce1}\begin{cases}
\Box w=(1-\zeta)\mathcal{N}(\partial u, \partial^2 u)-[\Box, \zeta]u,\qquad  \quad (t, x)\in \mathbb{R}_{+}\times\mo,\\[2mm]
\partial_{\boldsymbol{\nu}}w=0, \qquad  \qquad (t, x)\in\mathbb{R}_{+}\times\partial\mo,\\[2mm]
w(t, x)\equiv 0, \qquad \qquad  t\leq 0.
\end{cases}
\end{equation}
To prove Theorem \ref{them1} by the continuity induction argument, we 
require to establish the weighted estimate of $w$.

First, let $v$ be the solution of the following problem
\begin{equation}\label{reduce11}\begin{cases}
\Box v=-[\Box, \zeta]u,\qquad  \ (t, x)\in\mathbb{R}_{+}\times\mo,\\[2mm]
\partial_{\boldsymbol{\nu}}v=0,\qquad  \ (t, x)\in\mathbb{R}_{+}\times\partial\mo,\\[2mm]
v(t, x)\equiv 0,\qquad  \ t\leq 0.
\end{cases}
\end{equation}
By \eqref{Start1} and Theorem \ref{Point}, we can derive that there exists an absolute constant $\kappa_1>0$
such that
\begin{equation}\label{reduce4}
(1+t+|x|)\sum\limits_{\mu+|\alpha|\leq 60}|L^{\mu}Z^{\alpha}v(t, x)|\leq \kappa_1\varepsilon.
\end{equation}
Indeed, by Theorem \ref{Point},
\begin{equation*}\begin{aligned}
&(1+t+|x|)\sum\limits_{\mu+|\alpha|\leq 60}|L^{\mu}Z^{\alpha}v(t, x)|\\[2mm]
&\lesssim\int_0^{t}\int_{\{x\in\mo: (s, x)\in D_0\}}\sum\limits_{\mu+|\alpha|\leq 65\atop \mu\leq 61}|L^{\mu}Z^{\alpha}([\Box, \zeta]u)(s, x)|\frac{dxds}{|x|}\\[2mm]
&\quad +\int_0^{t}\sum\limits_{\mu+|\alpha|\leq 62\atop \mu\leq 61}
\|L^{\mu}\partial^{\alpha}([\Box, \zeta]u)(s, \cdot)\|_{L^2(|x|\leq 2)}ds\\
&=I_{0}+II_{0}.
\end{aligned}
\end{equation*}
With \eqref{Start1}-\eqref{Start2} and the pseudo homogeneity of $\zeta$, we arrive at
\begin{equation*}\begin{aligned}
I_{0}\lesssim& \varepsilon+\int_0^{t}\int_{4s\leq |x|\leq 8s\atop |x|\geq 8}\sum\limits_{\mu+|\alpha|\leq 65\atop \mu\leq 61}|L^{\mu}Z^{\alpha}([\Box, \zeta]u)(s, x)|\frac{dxds}{|x|}\\[2mm]
\lesssim&\int_0^{t}\int_{4s\leq |x|\leq 8s\atop |x|\geq 8}\sum\limits_{\mu+|\alpha|\leq 65\atop\mu\leq 61}\left(\frac{1}{|x|^2}|L^{\mu}Z^{\alpha}u|+\frac{1}{|x|}|L^{\mu}Z^{\alpha}\partial u|\right)(s, x)\frac{dxds}{|x|}\\[2mm]
\lesssim&\varepsilon\int_{0}^{t}\left(\int_{4s\leq |x|\leq 8s\atop |x|\geq 4}\frac{1}{|x|^6}dx\right)^{1/2}ds\\[2mm]
\lesssim&\varepsilon,\\[2mm]
II_{0}\lesssim&\varepsilon,
\end{aligned}
\end{equation*}
here the last estimate for $II_{0}$ comes from the property of function $\zeta$.

We now start to set up the continuity induction argument. Suppose that
\begin{equation}\label{induction1}
(1+t+r)\sum\limits_{|\alpha|\leq 24}(|Z^{\alpha}w(t, x)|+|Z^{\alpha}\partial w(t, x)|)\leq 4\kappa_1\varepsilon,
\end{equation}
where $\kappa_1$ is given in \eqref{reduce4}. To prove Theorem \ref{them1} by the continuity induction argument, we shall prove that for small $\ve>0$,
\begin{equation}\label{induction1.1}
(1+t+r)\sum\limits_{|\alpha|\leq 24}(|Z^{\alpha}w(t, x)|+|Z^{\alpha}\partial w(t, x)|)\leq 2\kappa_1\varepsilon.
\end{equation}
The proof process will be divided into the following three parts.

\vskip 0.2 true cm

{\bf Part I. Under assumption \eqref{induction1}, the following estimates hold:

 \begin{equation}\label{induction2}
(1+t+r)\sum\limits_{|\alpha|+\mu\leq 32\atop\mu\leq 2}|L^{\mu}Z^{\alpha}w(t, x)|\leq B_1\varepsilon (1+t)^{1/5}\ln (2+t),
\end{equation}
\begin{equation}\label{induction3}
\sum\limits_{|\alpha|\leq 68}\|\partial^{\alpha}\partial w(t, \cdot)\|\leq B_2\varepsilon (1+t)^{1/20},
\end{equation}
\begin{equation}\label{induction4}
\sum\limits_{|\alpha|+\mu\leq 47\atop \mu\leq 3}\|L^{\mu}Z^{\alpha}\partial w(t, \cdot)\|\leq B_3\varepsilon (1+t)^{1/10},
\end{equation}
and
\begin{equation}\label{induction5}
\sum\limits_{|\alpha|+\mu\leq 39\atop \mu\leq 3}\|\left<x\right>^{-1/2}L^{\mu}Z^{\alpha}\partial w\|_{L^2(S_t)}\leq B_4\varepsilon (1+t)^{1/10}(\ln(2+t))^{1/2}.
\end{equation}\\}

By \eqref{Null1}, we can rewrite the equation  in \eqref{Mpro} as
\begin{equation}\label{Proof1}
\Box_h u=\mathcal{S}(\partial u),
\end{equation}
where $h^{\alpha\beta}(\partial u)=-\mathcal{Q}^{\alpha\beta}(\partial u)$. Since $u=u_0+w$, it follows from \eqref{reduce2}, \eqref{induction1} and \eqref{notation} that
\begin{equation}\label{Proof2}
\|h(t)\|\lesssim \frac{\varepsilon}{1+t},\ \ \|\partial h(t)\|\lesssim \frac{\varepsilon}{1+t}.
\end{equation}
This, together with \eqref{timenl} in Lemma \ref{Timenl}, yields that for $0\leq M\leq 68$,
\begin{equation}\label{Proof3}
\partial_0\mathcal{E}_{0, M}^{1/2}(t)\lesssim \sum\limits_{0\leq j\leq M}\|\Box_h \partial_0^{j} u(t, \cdot)\|
+\frac{\varepsilon}{1+t}\mathcal{E}_{0, M}^{1/2}(t).
\end{equation}
Note that
\begin{equation}\label{Proof3.1}\begin{aligned}
&\sum\limits_{j\leq M}|\Box_h \partial_0^j u|\\[2mm]
\lesssim&\sum\limits_{j\leq M}\left(|[\Box_h, \partial_0^{j}]u|+|\partial_0^{j}\mathcal{S}(\partial u)|\right)\\[2mm]
\lesssim& (\sum\limits_{j\leq M}|\partial_0^j \partial u|+\sum\limits_{j\leq M-1}|\partial_0^j \partial^2 u|)\sum\limits_{|\alpha|\leq 24}|\partial^{\alpha}\partial u|+\sum\limits_{|\alpha|< M-24}|\partial^{\alpha}\partial u| \sum\limits_{24<|\beta|\leq M/2}|\partial^{\alpha}\partial u|\\[2mm]
\lesssim&\frac{\varepsilon}{1+t}(\sum\limits_{j\leq M}|\partial_0^j \partial u|+\sum\limits_{j\leq M-1}|\partial_0^j \partial^2 u|)+\sum\limits_{|\alpha|< M-24}|\partial^{\alpha}\partial u| \sum\limits_{24<|\beta|\leq M/2}|\partial^{\beta}\partial u|,
\end{aligned}
\end{equation}
here the cubic terms in $\mathcal{N}(\partial u, \partial^2 u)$ are neglected since \eqref{reduce2} and \eqref{induction1} imply $|\partial^{\alpha}\partial u|\lesssim\varepsilon/(1+t)$ for $|\alpha|\leq 24$. Furthermore,
by using Lemma \ref{Elliptic} and repeating the above argument, we have that for $0<\varepsilon\leq \varepsilon_0$,
\begin{equation}\label{Proofaux1}\begin{aligned}
&\sum\limits_{j\leq M-1}\|\partial_0^j\partial^2 u(t, \cdot)\|\\[2mm]
\lesssim& \sum\limits_{j\leq M}\|\partial_0^j\partial u(t, \cdot)\|+\sum\limits_{j\leq M-1}\|\partial_0^j\Box u(t, \cdot)\|\\[2mm]
\lesssim& \sum\limits_{j\leq M}\|\partial_0^j\partial u(t, \cdot)\|+\frac{\varepsilon}{1+t}\sum\limits_{j\leq M-1}\|\partial_0^j \partial u(t, \cdot)\|+\sum\limits_{|\alpha|< M-24\atop |\beta|\leq M/2}\|\partial^{\alpha}\partial u(t, \cdot)\partial^{\beta}\partial u(t, \cdot)\|.
\end{aligned}
\end{equation}
In addition, for small $\varepsilon>0$, it follows from \eqref{Proof2}-\eqref{Proofaux1} and \eqref{elcompare} that
\begin{equation}\label{Proof4}
\partial_0 \mathcal{E}_{0, M}^{1/2}(t)\lesssim\frac{\varepsilon}{1+t}\mathcal{E}_{0, M}^{1/2}(t)+\sum\limits_{|\alpha|< M-24\atop |\beta|\leq M/2}\|\partial^{\alpha}\partial u(t, \cdot)\partial^{\beta}\partial u(t, \cdot)\|.
\end{equation}
If $M=24$, the last term in \eqref{Proof4} does not appear. Then by Gronwall's inequality
and $\mathcal{E}_{0, M}^{1/2}(0)\lesssim \varepsilon$ due to \eqref{Start}, one has that
\begin{equation}\label{Proof5}
\sum\limits_{j\leq 24}\|\partial_0^{j}\partial u(t, \cdot)\|\leq 2 \mathcal{E}_{0, 24}^{1/2}(t)\lesssim \varepsilon (1+t)^{C\varepsilon}.
\end{equation}
Together with \eqref{Proofaux1}, elliptic regularity estimate \eqref{see2} in Lemma \ref{Conclusion} and the similar argument of \eqref{Proof3.1} for $\Box u$, \eqref{Proof5} yields
\begin{equation}\label{Proof6}
\sum\limits_{|\alpha|\leq 24}\|\partial^{\alpha}\partial u(t, \cdot)\|\lesssim \varepsilon (1+t)^{C\varepsilon}.
\end{equation}
If $24<M\leq 68$, we shall deal with the last term in \eqref{Proof4}. It follows from Lemma \ref{Sobolev} that
\begin{equation*}\label{Proof7}
\sum\limits_{|\alpha|\leq M-24\atop |\beta|\leq M/2}\|(\partial^{\alpha}\partial u\partial^{\beta}\partial u)(t, \cdot)\|\lesssim \sum\limits_{|\alpha|\leq\max(M-22, 2+M/2)}\|\left<x\right>^{-1/2}Z^{\alpha}\partial u(t, \cdot)\|^2.
\end{equation*}
This, together with \eqref{Start}, \eqref{Proof4} and Gronwall's inequality, yields
\begin{equation}\label{Proof8}
\mathcal{E}_{0, M}^{1/2}(t)\lesssim (1+t)^{C\varepsilon}\left(\varepsilon+\sum\limits_{|\alpha|\leq\max(M-22, 2+M/2)}\|\left<x\right>^{-1/2}Z^{\alpha}\partial u\|_{L^2(S_t)}^2\right).
\end{equation}

To treat $\mathcal{E}_{0, M}(t)$ better and finish the proof in {\bf Part I}, we require to establish the following Proposition.
\begin{prop}\label{Auxiliary} For given nonnegative integer $\mu_0\leq 3$, we define the integer $M\leq 68-8\mu_0$ in the following three inequalities respectively. If
\begin{equation}\label{auxil1}\begin{aligned}
&\sum\limits_{|\alpha|\leq M-3}\|\left<x\right>^{-1/2}L^{\mu_0}\partial^{\alpha}\partial u\|_{L^2(S_t)}+\sum\limits_{|\alpha|\leq M-4}\|L^{\mu_0}Z^{\alpha}\partial u(t, \cdot)\|+\sum\limits_{|\alpha|\leq M-6}\|\left<x\right>^{-1/2}L^{\mu_0}Z^{\alpha}\partial u\|_{L^2(S_t)}\\[2mm]
&\lesssim\varepsilon (1+t)^{C\varepsilon +\sigma}
\end{aligned}
\end{equation}
holds with some $\sigma>0$ for $\mu_0=0, 1, 2$, then there is a constant $C'>0$ such that
\begin{equation}\label{auxil2}\begin{aligned}
&\sum\limits_{\mu+|\alpha|\leq M-2\atop \mu\leq \mu_0}\|\left<x\right>^{-1/2}L^{\mu}\partial^{\alpha}\partial u\|_{L^2(S_t)}+\sum\limits_{\mu+|\alpha|\leq M-3\atop \mu\leq \mu_0}\|L^{\mu}Z^{\alpha}\partial u(t, \cdot)\|\\[2mm]
&\qquad+\sum\limits_{\mu+|\alpha|\leq M-5\atop \mu\leq \mu_0}\|\left<x\right>^{-1/2}L^{\mu}Z^{\alpha}\partial u\|_{L^2(S_t)}
\lesssim\varepsilon (1+t)^{C'\varepsilon+C'\sigma},\qquad \ \mu_0=0, 1, 2, 3,
\end{aligned}
\end{equation}
and
\begin{equation}\label{auxil1.1}
\sum\limits_{\mu+|\alpha|\leq M\atop \mu\leq \mu_0}\|L^{\mu}\partial^{\alpha}\partial u(t, \cdot)\|\lesssim \varepsilon (1+t)^{C'\varepsilon+C'\sigma},\qquad  \ \mu_0=0, 1, 2, 3.
\end{equation}

\end{prop}

We shall postpone the proof of Proposition \ref{Auxiliary} in the subsequent {\bf Part II}. Now
we give some illustrations for this Proposition. When $\mu_0=0$ in \eqref{auxil1},
starting from \eqref{Proof6}, by an induction argument on $M$, \eqref{auxil2} and \eqref{auxil1.1}  will stand for $\mu_0=0,1$ respectively.  Other left cases can be also obtained by induction argument both on $\mu_0$ and $M$.

Based on Proposition \ref{Auxiliary} and the fact that $u=w+u_0$ with $u_0$ satisfying \eqref{reduce2}, \eqref{induction3} comes from \eqref{auxil1.1} for the case $\mu_0=0$ and $C'\varepsilon+C'\sigma=1/20$;
\eqref{induction4} and \eqref{induction5} are derived from \eqref{auxil2} for the case $\mu_0=3$ and $C'\varepsilon+C'\sigma=1/10$.
In addition, \eqref{induction2} follows from Theorem \ref{Point}, \eqref{reduce2}, \eqref{reduce4} and \eqref{induction5} via a direct verification.

\vskip 0.2 true cm

{\bf Part II. The proof of Proposition \ref{Auxiliary}}

\vskip 0.2 true cm

It suffices to prove \eqref{auxil2} and \eqref{auxil1.1} for $\mu_0=0, 1$ since the other cases can be treated in the same way. The proof is divided into three steps.\vskip 0.2 true cm

 {\bf Step 1. Proof of \eqref{auxil2} for $\mu_0=0$}\vskip 0.2 true cm

Since $u=w+u_0$, by \eqref{reduce2} and \eqref{KKspace} in Proposition \ref{MKSS}, we have
\begin{equation}\label{auxil3}\begin{aligned}
&(\ln(2+t))^{-1/2}\sum\limits_{|\alpha|\leq M-2}\|\left<x\right>^{-1/2}\partial^{\alpha}\partial u\|_{L^2(S_t)}\\[2mm]
\lesssim&\varepsilon+(\ln(2+t))^{-1/2}\sum\limits_{|\alpha|\leq M-2}\|\left<x\right>^{-1/2}\partial^{\alpha}\partial w\|_{L^2(S_t)}\\[2mm]
\lesssim&\varepsilon +\sum\limits_{|\alpha|\leq M-1}\int_0^{t}\|\partial^{\alpha}\Box w\|ds+\sum\limits_{|\alpha|\leq M-2}\|\partial^{\alpha}\Box w\|_{L^2(S_t)}\\[2mm]
\lesssim &\varepsilon+\sum\limits_{|\alpha|\leq M-1}\int_0^{t}\|\partial^{\alpha}\Box u(s, \cdot)\| ds+\sum\limits_{|\alpha|\leq M-2}\|\partial^{\alpha}\Box u\|_{L^2(S_t)}.
\end{aligned}
\end{equation}
If $M\leq 24$, it follows from \eqref{reduce2}, \eqref{induction1} and  equation \eqref{Proof1} that
\begin{equation*}
\sum\limits_{|\alpha|\leq 23}|\partial^{\alpha}\Box u|\lesssim \sum\limits_{|\alpha|\leq 24}|\partial^{\alpha}\partial u|^2\lesssim\frac{\varepsilon}{1+t}\sum\limits_{|\alpha|\leq 24}|\partial^{\alpha}\partial u|.
\end{equation*}
Combining this with \eqref{Proof6} and \eqref{auxil3} yields
\begin{equation}\label{auxil4}
(\ln(2+t))^{-1/2}\sum\limits_{|\alpha|\leq M-2}\|\left<x\right>^{-1/2}\partial^{\alpha}\partial u\|_{L^2(S_t)}\lesssim \varepsilon (1+t)^{C\varepsilon+\sigma}.
\end{equation}
If $24<M\leq 68$, by repeating the argument for \eqref{Proof8} and applying \eqref{auxil1} for $\mu_0=0$, we then arrive at
\begin{equation*}\begin{aligned}
&(\ln(2+t))^{-1/2}\sum\limits_{|\alpha|\leq M-2}\|\left<x\right>^{-1/2}\partial^{\alpha}\partial u\|_{L^2(S_t)}\\[2mm]
\lesssim& \varepsilon (1+t)^{2C\varepsilon+2\sigma}+\sum\limits_{|\alpha|\leq\max(M-22, 2+M/2)}\|\left<x\right>^{-1/2}Z^{\alpha}\partial u\|_{L^2(S_t)}^2\\[2mm]
&+\sup\limits_{0\leq s\leq t}\left(\sum\limits_{|\alpha|\leq M-6}\|Z^{\alpha}\partial u(s, \cdot)\|\right)\sum\limits_{|\alpha|\leq\max(M-22, 2+M/2)}\|\left<x\right>^{-1/2}Z^{\alpha}\partial u\|_{L^2(S_t)}\\[2mm]
\lesssim& \varepsilon (1+t)^{2C\varepsilon+2\sigma}.
\end{aligned}
\end{equation*}
Combining this with \eqref{auxil4} yields for $M\leq 68$
\begin{equation}\label{auxil5}
\sum\limits_{|\alpha|\leq M-2}\|\left<x\right>^{-1/2}\partial^{\alpha}\partial u\|_{L^2(S_t)}\lesssim \varepsilon (1+t)^{2C\varepsilon+2\sigma}\ln(2+t).
\end{equation}
This derives the estimate of the first term in the left hand side of \eqref{auxil2}.

Next, we estimate the second term in the left hand side of \eqref{auxil2}.
For $E_{\mu, \nu}(t)$ defined by \eqref{BE}, with the help of \eqref{reduce2} and \eqref{induction1} which can be used
to deal with the cubic terms in $\mathcal{N}(\partial u, \partial^2 u)$, we have
\begin{equation}\label{5.1}\begin{aligned}
&\sum\limits_{|\alpha|\leq M-3}\|\Box_h Z^{\alpha}u(t, \cdot)\|\\[2mm]
\lesssim&\sum\limits_{|\alpha|+|\beta|\leq M-3}\|Z^{\alpha}\partial u(t, \cdot)Z^{\beta}\partial u(t, \cdot)\|+\sum\limits_{|\alpha|+|\beta|\leq M-3\atop |\alpha|\geq 1}\|Z^{\alpha}\partial u(t, \cdot)Z^{\beta}\partial^2 u(t, \cdot)\|\\[2mm]
\lesssim&\sum\limits_{|\alpha|\leq M-3\atop |\beta|\leq 24}\|Z^{\alpha}\partial u(t, \cdot)\| \|Z^{\beta}\partial u(t, \cdot)\|_{\infty}+\sum\limits_{|\alpha|\leq M-27;\ |\beta|\leq M-27}\|Z^{\alpha}\partial u(t, \cdot)Z^{\beta}\partial u(t, \cdot)\|\\[2mm]
\lesssim&\frac{\varepsilon}{1+t}E_{0, M-3}^{1/2}(t)+\sum\limits_{|\alpha|\leq M-25}\|\left<x\right>^{-1/2}Z^{\alpha}\partial u(t, \cdot)\|^2,
\end{aligned}
\end{equation}
here we point out that the last inequality comes from \eqref{reduce2}, \eqref{induction1} and Lemma \ref{Sobolev}.
Substituting \eqref{5.1} into \eqref{rotationnl} derives
\begin{equation*}\begin{aligned}
\partial_0 E_{0, M-3}(t)\lesssim&\frac{\varepsilon}{1+t}E_{0, M-3}(t)
+\sum\limits_{|\alpha|\leq M-25}\|\left<x\right>^{-1/2}Z^{\alpha}\partial u(t, \cdot)\|^2 E_{0, M-3}^{1/2}(t)\\[2mm]
&+\sum\limits_{|\alpha|\leq M-2}\|\left<x\right>^{-1/2}\partial^{\alpha}\partial u(t, \cdot)\|^2.
\end{aligned}
\end{equation*}
Then it follows from Gronwall's inequality, \eqref{Start} and \eqref{auxil5} that
\begin{equation}\label{Proof8.1}\begin{aligned}
\sum\limits_{|\alpha|\leq M-3}\|Z^{\alpha}\partial u(t, \cdot)\|^2\lesssim& E_{0, M-3}(t)\\[2mm]
\lesssim& (1+t)^{C\varepsilon}\biggl(\varepsilon^2+\sum\limits_{|\alpha|\leq M-25}\|\left<x\right>^{-1/2}Z^{\alpha}\partial u(t, \cdot)\|_{L^2(S_t)}^2\sup\limits_{0<s<t}E_{0, M-3}^{1/2}(s)\\[2mm]
&+\sum\limits_{|\alpha|\leq M-2}\|\left<x\right>^{-1/2}\partial^{\alpha}\partial u(t, \cdot)\|_{L^2(S_t)}^2\biggr)\\[2mm]
\lesssim& \varepsilon (1+t)^{C'\varepsilon+C'\sigma}.
\end{aligned}
\end{equation}
Here, the last inequality comes from \eqref{auxil1} for $\nu=0$. Analogously, the estimate on the third term in \eqref{auxil2} comes from \eqref{KKrotation}, \eqref{Proof8.1} and \eqref{auxil1} for $\mu_0=0$.\qed
\vskip 0.2 true cm

{\bf Step 2. Proof of \eqref{auxil1.1} for $\mu_0=0$}\vskip 0.2 true cm

It derives from \eqref{Proof8} and \eqref{auxil2} for $\mu_0=0$ that for arbitrarily $\sigma>0$,
\begin{equation}\label{Proof9}
\mathcal{E}_{0, M}^{1/2}(t)\lesssim \varepsilon (1+t)^{C\varepsilon+\sigma}.
\end{equation}
In addition, based on \eqref{auxil2} for $\mu_0=0$ and \eqref{Proof9}, by Lemma \ref{Elliptic} with the similar analysis as \eqref{Proof3.1}, one obtains \begin{equation*}\sum\limits_{|\alpha|\leq M}\|\partial^{\alpha}\partial u(t, \cdot)\|\lesssim\varepsilon (1+t)^{C\varepsilon+\sigma}.
\end{equation*}
This shows \eqref{auxil1.1} for $\mu_0=0$.\qed

\vskip 0.2 true cm

{\bf Step 3. Proof of \eqref{auxil2} and \eqref{auxil1.1} for $\mu_0=1$}\vskip 0.2 true cm

 At first, we establish a suitable version
of  \eqref{spacenl5} for $N_0+\mu_0\leq 60$ and $\mu_0=1$. It is noticed that for $M\leq 60$,
\begin{equation*}\begin{aligned}
&\sum\limits_{\mu+\alpha\leq M\atop \mu\leq 1}\left(|\mathcal{L}^{\mu}\partial_0^{\alpha}\Box_h u|+|[\mathcal{L}^{\mu}\partial_0^{\alpha}, h^{\alpha\beta}\partial_{\alpha\beta}^2]u|\right)\\[2mm]
\lesssim&\left(\sum\limits_{|\alpha|\leq M-1}|\mathcal{L}\partial_0^{\alpha}\partial u|+\sum\limits_{|\alpha|\leq M-2}|\mathcal{L}\partial_0^{\alpha}\partial^2 u|\right)\sum\limits_{|\beta|\leq 24}|\partial^{\beta}\partial u|\\[2mm]
&+\sum\limits_{|\alpha|\leq M-25}|L\partial^{\alpha}\partial u| \sum\limits_{|\beta|\leq M}|\partial^{\beta}\partial u|+\sum\limits_{|\alpha|\leq M}|\partial^{\alpha}\partial u|\sum\limits_{|\beta|\leq\max(M-24, M/2)}|\partial^{\beta}\partial u|.
\end{aligned}
\end{equation*}
Combining this with \eqref{reduce2}, \eqref{induction1}, Lemma \ref{Sobolev} and Lemma \ref{Elliptic} yields that for $M\leq 60$,
\begin{equation*}\begin{aligned}
&\sum\limits_{\mu+|\alpha|\leq M\atop \mu\leq 1}\left(\|\mathcal{L}\partial_0^{\alpha}\Box_h u(t, \cdot)\|+\|[\mathcal{L}\partial_0^{\alpha}, h^{\alpha\beta}\partial_{\alpha\beta}^2]u(t, \cdot)\|\right)\\[2mm]
\lesssim&\frac{\varepsilon}{1+t}\sum\limits_{\mu+|\alpha|\leq M\atop\mu\leq 1}\|\mathcal{L}\partial_0^{\alpha}\partial u(t, \cdot)\|+\sum\limits_{|\alpha|\leq\max(M, 2+M/2)}\|\left<x\right>^{-1/2}Z^{\alpha}\partial u(t, \cdot)\|^2\\[2mm]
&+\sum\limits_{|\alpha|\leq M-25}\|\left<x\right>^{-1/2}L\partial^{\alpha}\partial u(t, \cdot)\|\sum\limits_{|\alpha|\leq 34}\|\left<x\right>^{-1/2}Z^{\alpha}\partial u(t, \cdot)\|.
\end{aligned}
\end{equation*}
Based on this, for small $\varepsilon>0$,  \eqref{spacenl5} holds with $\delta=C\varepsilon$ and
$$H_{1, M-1}(t)=\sum\limits_{|\alpha|\leq M-25}\|\left<x\right>^{-1/2}L\partial^{\alpha}\partial u(t, \cdot)\|^2+\sum\limits_{|\alpha|\leq 62}\|\left<x\right>^{-1/2}Z^{\alpha}\partial u(t, \cdot)\|^2.$$

Since $\mathcal{E}_{\mu, \nu}(0)\lesssim \varepsilon$ for $\mu+\nu\leq 68$ by \eqref{Start}, it follows from \eqref{spacenl6} and \eqref{auxil2} for $\mu_0=0$
that for $M\leq 60$,
\begin{equation}\label{Proof12}\begin{aligned}
&\sum\limits_{\mu+|\alpha|\leq M\atop \mu\leq 1}\|L^{\mu}\partial^{\alpha}\partial u(t, \cdot)\|\\[2mm]
\lesssim&\varepsilon (1+t)^{C\varepsilon+\sigma}+(1+t)^{C\varepsilon}\sum\limits_{|\alpha|\leq M-25}\|\left<x\right>^{-1/2}L\partial^{\alpha}\partial u\|_{L^2(S_t)}^2\\[2mm]
&+(1+t)^{C\varepsilon}\int_0^{t}\sum\limits_{|\alpha|\leq M+1}\|\partial^{\alpha}
\partial u(s, \cdot)\|_{L^2(|x|\leq 1)}ds.
\end{aligned}
\end{equation}
By \eqref{spacelo2}, one has that
\begin{equation}\label{5.2}
\int_0^{t}\sum\limits_{|\alpha|\leq M+1}\|\partial^{\alpha}\partial w(s, \cdot)\|_{L^2(|x|\leq 2)}ds
\lesssim\sum\limits_{|\alpha|\leq M+2}\int_0^t(\int_0^s\|\partial^{\alpha}
\Box w(\tau, \cdot)\|_{L^2(\{x: ||x|-(s-\tau)|<10\}\cap\mo)}d\tau) ds.
\end{equation}
In addition, it follows from \eqref{reduce2} that
\begin{equation}\label{5.3}\begin{aligned}
&\sum\limits_{|\alpha|\leq M+1}\int_0^{t}\|\partial^{\alpha}\partial u(s, \cdot)\|_{L^2(|x|\leq 2)}ds\\[2mm]
\lesssim&\varepsilon\ln(2+t)+\sum\limits_{|\alpha|\leq M+2}\int_0^{t}(\int_0^{s}\|\partial^{\alpha}\Box u(\tau, \cdot)\|_{L^2(\{x: ||x|-(s-\tau)|<10\}\cap\mo)}d\tau)ds.
\end{aligned}
\end{equation}
Note that
\begin{equation*}
\sum\limits_{|\alpha|\leq M+2}|\partial^{\alpha}\Box u|\lesssim \sum\limits_{|\alpha|\leq M+3}|\partial^{\alpha}\partial u|\sum\limits_{|\alpha|\leq 1+M/2}|\partial^{\alpha}\partial u|.
\end{equation*}
This, together with Lemma \ref{Sobolev} and the fact of $M\leq 60$, yields
\begin{equation}\label{5.4}
\sum\limits_{|\alpha|\leq M+2}\|\partial^{\alpha}\Box u(\tau, \cdot)\|_{L^2(\{x:||x|-(s-\tau)|<10\}\cap\mo)}
\lesssim \sum\limits_{|\alpha|\leq 63}\|\left<x\right>^{-1/2}Z^{\alpha}\partial u\|_{L^2(\{x: ||x|-(s-\tau)|<20\}\cap\mo)}^2.
\end{equation}
Thus for $M\leq 60$, by \eqref{auxil2} in Lemma \ref{Auxiliary} and \eqref{5.3}-\eqref{5.4}, we have
\begin{equation*}\begin{aligned}
\sum\limits_{|\alpha|\leq M+1}\int_0^{t}\|\partial^{\alpha}\partial u(s, \cdot)\|_{L^2(|x|\leq 1)}ds\lesssim& \varepsilon\ln(2+t)+\sum\limits_{|\alpha|\leq 63}\|\left<x\right>^{-1/2}Z^{\alpha}\partial u\|_{L^2(S_t)}^2\\[2mm]
&\lesssim \varepsilon (1+t)^{C\varepsilon+\sigma}.
\end{aligned}
\end{equation*}
Meanwhile, by \eqref{Proof12} and \eqref{5.2}-\eqref{5.4}, we arrive at
\begin{equation}\label{Proof12.1}\begin{aligned}
&\sum\limits_{\mu+|\alpha|\leq M\atop\mu\leq 1}\|L^{\mu}\partial^{\alpha}\partial u(t, \cdot)\|\\[2mm]
\lesssim&\varepsilon (1+t)^{C\varepsilon +\sigma}
+(1+t)^{C\varepsilon}\sum\limits_{|\alpha|\leq M-25}\|\left<x\right>^{-1/2}L\partial^{\alpha}\partial u\|_{L^2(S_t)}^2.
\end{aligned}
\end{equation}
This gives the desired bounds for $M\leq 24$.

If we utilize \eqref{KKspace} with $\mu_0=1$ and $N_0+\mu_0=60$, then by analogous proof of Lemma \ref{Auxiliary}
when $M=68$ is replaced by $M=60$ and $u$ is replaced by $Lu$ respectively, we can get
\begin{equation}\label{Proof13}\begin{aligned}
&\sum\limits_{\mu+|\alpha|\leq 58\atop \mu\leq 1}\|\left<x\right>^{-1/2}L^{\mu}\partial^{\alpha}\partial w\|_{L^2(S_t)}
+\sum\limits_{\mu+|\alpha|\leq 57\atop \mu\leq 1}\|L^{\mu}Z^{\alpha}\partial w(t, \cdot)\|\\
&+\sum\limits_{\mu+|\alpha|\leq 55\atop \mu\leq 1}\|\left<x\right>^{-1/2}L^{\mu}Z^{\alpha}\partial w\|_{L^2(S_t)}\\[2mm]
\lesssim&\varepsilon (1+t)^{C\varepsilon+C\sigma}.
\end{aligned}
\end{equation}
This is just only \eqref{auxil2} of $\mu_0=1$.
Consequently, it follows from \eqref{Proof12.1} and \eqref{Proof13} that \eqref{auxil1.1} stands for $\mu_0=1$.
\vskip 0.2 true cm

{\bf Part III. Proof of \eqref{induction1.1}}
\vskip 0.2 true cm

By Theorem \ref{Point} and \eqref{reduce1}-\eqref{reduce4}, in order to prove \eqref{induction1.1},
it suffices to show
\begin{equation}\label{proof1}
I+II\lesssim \varepsilon^2,
\end{equation}
where
\begin{equation*}\begin{aligned}
I&=\int_0^{t}\int_{\mo}\sum\limits_{\nu+|\beta|\leq 29\atop \nu\leq 1}|L^{\nu}Z^{\beta}\mathcal{N}(\partial u, \partial^2 u)(s, y)|\frac{dy ds}{|y|},\\[2mm]
II&=\int_0^{t}\sum\limits_{\nu+|\beta|\leq 26\atop \nu\leq 1}\|L^{\nu}\partial^{\beta}\mathcal{N}(\partial u, \partial^2 u)(s, \cdot)\|_{L^2(|x|\leq 2)}ds.
\end{aligned}
\end{equation*}
By \eqref{Null1}, we write
\begin{equation}\label{proof2}
\mathcal{N}(\partial u, \partial^2 u)=\mathcal{S}^{\alpha\beta}\partial_{\alpha}u\partial_{\beta}u+\mathcal{Q}_{\mu}^{\alpha\beta}\partial_{\mu}u\partial_{\alpha\beta}^2u
+\mathcal{R}(\partial u, \partial^2 u),
\end{equation}
where $\mathcal{R}(\partial u, \partial^2 u)$ is the cubic term and linear in $\p^2 u$. Making use of Lemma \ref{Null},
one has
\begin{equation}\label{proof3}\begin{aligned}
&\sum\limits_{\nu+|\beta|\leq 29\atop\nu\leq 1}|L^{\nu}Z^{\beta}\left(\mathcal{S}^{\alpha\gamma}\partial_{\alpha}u\partial_{\gamma}u
+\mathcal{Q}_{\mu}^{\alpha\gamma}\partial_{\mu}u\partial_{\alpha}\partial_{\gamma}u\right)|\\[2mm]
\lesssim&\frac{1}{|y|}\sum\limits_{\mu+|\alpha|\leq 31\atop \mu\leq 2}|L^{\mu}Z^{\alpha}u|\sum\limits_{\mu+|\alpha|\leq 31\atop \mu\leq 2}|L^{\mu}Z^{\alpha}\partial u|\\[2mm]
&+\frac{\left<t-r\right>}{\left<t+r\right>}\sum\limits_{\mu+|\alpha|\leq 29\atop \mu\leq 1}|L^{\mu}Z^{\alpha}\partial u|\sum\limits_{\mu+|\alpha|\leq 29\atop \mu\leq 1}|L^{\mu}Z^{\alpha}\partial^2 u|\\[2mm]
=&I_1+I_2.
\end{aligned}
\end{equation}
By \eqref{induction2} and \eqref{induction5} with \eqref{reduce2}, for $0<\delta<1/10$,
\begin{equation}\label{proof4}\begin{aligned}
\int_0^{t}\int_{\mo}I_1\frac{dy ds}{|y|}
\leq& C(\delta)\varepsilon\int_0^{t}\sum\limits_{\mu+|\alpha|\leq 32\atop \mu\leq 2}\|\left<\cdot\right>^{-1/2}L^{\mu}Z^{\alpha}\partial u(s, \cdot)\|\left<s\right>^{-4/5+\delta}ds\\[2mm]
=:& C(\delta)\ve I_{11},
\end{aligned}
\end{equation}
where the first inequality comes from the Young inequality together with \eqref{induction5} and \eqref{reduce2}.
Direct computation yields
\begin{equation}\label{Proof4.0}\begin{aligned}
I_{11}^2\lesssim&\left<t\right>^{-2/5}\sum\limits_{\mu+|\alpha|\leq 32\atop\mu\leq 2}\|\left<\cdot\right>L^{\mu}Z^{\alpha}\partial u(s, \cdot)\|_{L^2(S_t)}^2+\int_0^{t}\sum\limits_{\mu+|\alpha|\leq 32\atop\mu\leq 2}\left<s\right>^{-7/5}\|\left<\cdot\right>^{-1/2}L^{\mu}Z^{\alpha}\partial u\|_{L^2(S_s)}^2 ds\\[2mm]
\lesssim&\varepsilon^2.
\end{aligned}
\end{equation}
Substituting this into \eqref{proof4} yields that
\begin{equation}\label{proof4.1}
\int_0^{t}\int_{\mo}I_1\frac{dyds}{|y|}\lesssim \varepsilon^2.
\end{equation}

In addition, we have that from the proof of Lemma \ref{Inter},
\begin{equation*}\begin{aligned}
&\int_0^{t}\int_{\mo}I_2\frac{dy ds}{|y|}\\[2mm]
\lesssim&\int_0^{t}\left<s\right>^{-1}(\sum\limits_{\mu+|\alpha|\leq 31\atop \mu\leq 2}\|L^{\mu}Z^{\alpha}\partial u(s, \cdot)\|+\sum\limits_{\mu+|\alpha|\leq 30\atop \mu\leq 1}\|\left<t+r\right>L^{\mu}Z^{\alpha}\Box u(s, \cdot)\|)\\[2mm]
&\quad \times\sum\limits_{\mu+|\alpha|\leq 30\atop \mu\leq 1}\|\left<y\right>^{-1}L^{\mu}Z^{\alpha}\partial u(s, \cdot)\|ds\\[2mm]
&+\int_0^{t}\sum\limits_{R=2^{\kappa}<t/2}\sum\limits_{\mu+|\alpha|\leq 30\atop \mu\leq 1}\|\left<y\right>^{-1}(L^{\mu}Z^{\alpha}\partial u, L^{\mu}Z^{\alpha}u)(s, \cdot)\|_{L^2(R/4\leq |x|\leq 2R)}\\[2mm]
&\quad \times\sum\limits_{\mu+|\alpha|\leq 30\atop \mu\leq 1}
\|\left<y\right>^{-1}L^{\mu}Z^{\alpha}\partial u(s, \cdot)\|_{L^2(R/2\leq |x|\leq R)}\\[2mm]
=&I_{21}+I_{22}.
\end{aligned}
\end{equation*}
Similar to the estimate for $I_1$, it follows from \eqref{reduce2}, \eqref{induction2} and \eqref{induction5} that
\begin{equation}\label{proof5}\begin{aligned}
I_{22}
\lesssim&\varepsilon \int_0^{t}(\ln(2+s))^2\left<s\right>^{-4/5}\sum\limits_{\mu+|\alpha|\leq 30\atop \mu\leq 1}\|\left<y\right>^{-1/2}L^{\mu}Z^{\alpha}\partial u(s, \cdot)\| ds\\[2mm]
\lesssim&\varepsilon\sum\limits_{\mu+|\alpha|\leq 30\atop \mu\leq 1}\|\left<s\right>^{-7/25}\left<y\right>^{-1}L^{\mu}Z^{\alpha}\partial u(s, y)\|_{L^2{(S_t)}}\\[2mm]
\lesssim& \varepsilon^2,
\end{aligned}
\end{equation}
here the last inequality is derived as for the estimate in \eqref{Proof4.0}.

In the similar way, it follows from \eqref{reduce2}, \eqref{induction2} and \eqref{induction4}-\eqref{induction5}  that
\begin{equation}\label{proof6}\begin{aligned}
I_{21}\lesssim& \varepsilon\int_0^{t}\left<s\right>^{-1+1/10}(1+\ln(2+s)\left<s\right>^{1/5})\sum\limits_{\mu+|\alpha|\leq 30\atop \mu\leq 1}\|\left<y\right>^{-1}L^{\mu}Z^{\alpha}\partial u(s, \cdot)\|ds\\[2mm]
\lesssim& \varepsilon^2
\end{aligned}
\end{equation}
holds by applying the fact that
\begin{equation*}\begin{aligned}
&\left(s+r\right>\sum\limits_{\mu+|\alpha|\leq 30\atop \mu\leq 1}|L^{\mu}Z^{\alpha}\Box u|\lesssim \left<s+r\right>\sum\limits_{\mu+|\alpha|\leq 31\atop\mu\leq 1}|L^{\mu}Z^{\alpha}\partial u|^2\\[2mm]
&\qquad \lesssim \varepsilon \ln(2+s) (1+s)^{1/5}\sum\limits_{\mu+|\alpha|\leq 31\atop \mu\leq 1}|L^{\mu}Z^{\alpha}\partial u|.
\end{aligned}
\end{equation*}

Based on Theorem \ref{Point}, \eqref{reduce2} and \eqref{induction2}-\eqref{induction5}, it is clear that the cubic term $\mathcal{R}(\partial u, \partial^2 u)$ in \eqref{proof2} admits the following estimates
\begin{equation*}
\int_0^{t}\int_{\mo}\sum\limits_{\mu+|\alpha|\leq 30\atop \mu\leq 1}|L^{\mu}Z^{\alpha}\mathcal{R}(\partial u, \partial^2 u)|(s, y)\frac{dy ds}{|y|}\lesssim \varepsilon^2.
\end{equation*}
Combining this with \eqref{proof2}-\eqref{proof6} yields that $I$ satisfies estimate \eqref{proof1}. We now deal with $II$ in \eqref{proof1}.
It follows from direction verification and \eqref{induction2} that
\begin{equation*}
II\lesssim\varepsilon\int_0^{t}\left<s\right>^{-4/5}\ln(2+s)\sum\limits_{\mu+|\alpha|\leq 31\atop \mu\leq 1}\|L^{\mu}Z^{\alpha}\partial u(s, \cdot)\|_{L^2(|y|\leq 2)}ds\lesssim \varepsilon^2.
\end{equation*}
This, together with the estimate of $I$, shows that \eqref{induction1.1} can be derived by \eqref{induction2}-\eqref{induction5}.
Thus  the proof of Theorem \ref{them1} is completed by the continuity induction argument together with the local existence of
problem \eqref{Mpro} in Sect.\ref{VI} below.

\section{Local existence of problem \eqref{Mpro}}\label{VI}

In this section, we establish the local existence of solution to general $n$-dimensional problem \eqref{Mpro} for $n\geq 2$
under  ``admissible condition'' \eqref{AC} and the compatible conditions of
the initial-boundary values. In addition, we assume $\mo$ is the exterior domain of $n$-dimensional
compact convex domain $\mathcal{K}$ with smooth boundary.

Using the notation $J_k u=\{\nabla^{\alpha} u: 1\leq |\alpha|\leq k\}$,
problem \eqref{Mpro}  can be rewritten as
\begin{equation}\label{Jpro}\begin{cases}
\Box u=\mathcal{Q}^{\alpha\beta}(\partial_0 u, J_1 u)\partial_{\alpha\beta}^2u+\mathcal{S}(\partial_0 u, J_1 u),\quad \ (t, x)\in\mathbb{R}_{+}\times\mo,\\[2mm]
\partial_{\boldsymbol{\nu}}u=0,\qquad \qquad  \qquad \qquad \qquad \qquad \qquad  (t, x)\in\mathbb{R}_{+}\times\partial\mo,\\[2mm]
(u, \partial_t u)(0, x)=\varepsilon (u_0, u_1)(x),\qquad\qquad   \ x\in\mo.
\end{cases}
\end{equation}
In addition, it is  assumed that for some fixed positive constant $M$ and integer $s>6+3n/2$,
\begin{equation}\label{JInitial}
\|u_0\|_{H^{s+1}}+\|u_1\|_{H^{s}}\leq M.
\end{equation}

\subsection{Description of compatible conditions}\label{VI.1}

Due to the smallness of solution $u$ and the form
of $\mathcal{Q}^{\alpha\beta}(\partial_0 u, J_1 u)$ in \eqref{Null1}, without loss of generality,
it is assumed that
\begin{equation}\label{assume}
\mathcal{Q}^{00}=0,\ \ \sum\limits_{\alpha, \beta=0}^{n}|\mathcal{Q}^{\alpha\beta}|\leq 1/2.
\end{equation}
In this situation, the equation in \eqref{Jpro} has the form
\begin{equation*}
\partial_0^2 u=\mathcal{F}(J_1\partial_0 u, J_2 u)
\end{equation*}
for certain function $\mathcal{F}$ smooth in its arguments. Set $\psi_0=\varepsilon u_0, \psi_1=\varepsilon u_1$ and $\psi_2=\mathcal{F}(J_2 u_0, J_1 u_1)$.
As shown in Sect.9 of \cite{KSS2}, for $k\in\Bbb N$, there exists a compatible function sequence $\{\psi_k\}$ such that
the smooth solution $u$ to problem \eqref{Jpro} satisfies
\begin{equation}\label{psi}
\partial_0^k u(0, x)=\psi_k\equiv\psi_k(J_k u_0, J_{k-1}u_1).
\end{equation}
Then the compatible conditions to problem \eqref{Jpro} can be stated as:

\begin{defn}\label{Definition} Under assumption \eqref{assume}, the compatible conditions
for problem \eqref{Jpro} are called to be of $s-$order if $\partial_{\boldsymbol{\nu}}\psi_j$ vanishes
on $\mathbb{R}_{+}\times\partial\mo$ for all $0\leq j\leq s-1$.
\end{defn}

It is well known that the compatible conditions are necessary to obtain
the local existence of smooth solutions to the initial-boundary value problem \eqref{Jpro}.
The following result is helpful to choose the iteration scheme for
establishing the local existence of \eqref{Jpro}.

\begin{lem}\label{Compatible} {\bf (See Lemma 9.1 and Lemma 9.3 in \cite{KSS2})}
Under assumption \eqref{JInitial} for $s>6+3n/2$ and the compatible conditions
of $s-$order, let $\{\psi_k(J_k u_0, J_{k-1}u_1)\}$ be the compatible functions for problem \eqref{Jpro}
introduced in \eqref{psi}, then one has

\noindent {\bf (A)} $\psi_k(J_k u_0, J_{k-1}u_1)\in H^{s+1-k}(\mo)$ for $0\leq k\leq s+1$ and
\begin{equation*}
\sum\limits_{0\leq k\leq s+1}\|\psi_k(J_k u_0, J_{k-1}u_1)\|_{H^{s+1-k}}\lesssim\varepsilon \left(\|u_0\|_{H^{s+1}}+\|u_1\|_{H^{s}}\right).
\end{equation*}

\noindent {\bf (B)} Suppose that $v(t, x)$ is a function such that for some $T>0$
and $0\leq j\leq s+1$,
\begin{equation*}
\partial_0^j v\in C([0, T); H^{s+1-j}(\mo)),
\end{equation*}
and for $0\leq k\leq s+1$,
\begin{equation*}
\partial_0^k v(0, \cdot)=\psi_k(J_k u_0, J_{k-1}u_1),\ \partial_{\boldsymbol{\nu}}v=0\ \ \text{on}\ \ \mathbb{R}_{+}\times\partial\mo.
\end{equation*}
Let $\bar\psi_j$ be the compatible function of $j-$order for the problem
\begin{equation}\label{Linear}\begin{cases}
\Box u=\mathcal{Q}^{\alpha\beta}(\partial_0 v, J_1 v)\partial_{\alpha\beta}u+\mathcal{S}(\partial_0 v, J_1 v),\ \ (t, x)\in\mathbb{R}_{+}\times\mo,\\[2mm]
\partial_{\boldsymbol{\nu}}u=0,\qquad\qquad  \qquad\qquad \ (t, x)\in\mathbb{R}_{+}\times\partial\mo,\\[2mm]
u(0, x)=u_0,\ \partial_t u(0, x)=u_1,\qquad  x\in\mo.
\end{cases}
\end{equation}
Then for $0\leq k\leq s+1$,
\begin{equation*}
\bar\psi_k=\psi_k(J_k u_0, J_{k-1}u_1).
\end{equation*}

\end{lem}

\subsection{Local existence of smooth solutions}

\begin{thm}\label{Lexistence} For problem \eqref{Jpro} with the initial data $(u_0, u_1)$
satisfying \eqref{JInitial}, if both the admissible condition \eqref{AC} and the  compatibility  conditions of order $s+1$ are satisfied,
then there exists a constant $\varepsilon_0>0$ and a constant $T^*>0$
depending on $s$ and $\varepsilon_0$ such that when $\varepsilon<\varepsilon_0$,
there exists a solution $u$ to problem \eqref{Jpro} on $[0, T^*]\times \mo$, which satisfies
\begin{equation}\label{lexistence}
\sup\limits_{0\leq t\leq T*}\sum\limits_{0\leq j\leq s+1}\|\partial_0^j u(t, \cdot)\|_{H^{s+1-j}}\lesssim \varepsilon \left(\|u_0\|_{H^{s+1}}+\|u_1\|_{H^{s}}\right).
\end{equation}

\end{thm}
To prove Theorem \ref{Lexistence}, based on the existence result of linear Neumann-wave problem (see \cite{IM} and \cite{YN}),
we shall establish some estimates for the linear problem \eqref{Linear}. For  $T>0$, define
\begin{equation*}
M_{s+1}(t,v)=\sum\limits_{|\alpha|\leq s+1}\|\partial^{\alpha}v(t, \cdot)\|,\ \ M_{s+1}[v]=\sup\limits_{0\leq t\leq T}M_{s+1}(v, t)
\end{equation*}
and
\begin{equation*}
\|q_v(t)\|=\sum\limits_{\alpha, \beta=0}^{n}\|\mathcal{Q}^{\alpha\beta}(\partial_0 v, J_1 v)(t, \cdot)\|_{\infty},\ \ \|\partial q_v(t)\|=\sum\limits_{\alpha, \beta=0}^{n}\|\partial \mathcal{Q}^{\alpha\beta}(\partial_0 v, J_1 v)(t, \cdot)\|_{\infty}.
\end{equation*}

\begin{lem}\label{Basice} Under the assumptions of Lemma \ref{Compatible}
and  the admissible condition \eqref{AC}, let $u$ be a solution of problem \eqref{Linear} with
\begin{equation*}
\partial_0^j u\in C([0, T); H^{s+1-j}(\mo)).
\end{equation*}
In addition, we suppose for some $M_0>0$
\begin{equation*}
\ds\sup_{0\le t\le T}\|q_v(t)\|\leq M_0\varepsilon.
\end{equation*}
Then there exists a constant $C>0$ independent of $u$ and $v$, such that for $0\leq t\leq T$,
\begin{equation}\label{GJ1}\begin{aligned}
&M_{s+1}(t,u)\\[2mm]
\leq& Ce^{CM_{s+1}[v]t}\biggl(M_{s+1}(0,u)+C(M_{s+1}[v])M_{s+1}[v]\int_0^{t}M_{s+1}(\sigma, u)d\sigma\\[2mm]
&+C(M_{s}[v])\int_0^{t}M_{s}^2(l, v)dl\biggr)+C(M_s(t, v))M_s(t,v)(M_{s+1}(t,u)+M_s(t,v)).
\end{aligned}
\end{equation}

\end{lem}

\begin{proof} By \eqref{AC} and Lemma \ref{Compatible} (B), one has
that for any function $w$ satisfying $\partial_{\boldsymbol{\nu}}w=0$ on $\mathbb{R}_{+}\times\partial\mo$,
\begin{equation}\label{JAC}
\mathcal{Q}^{\alpha\beta}(\partial_0 v, J_1 v)\boldsymbol{\nu}^{\alpha}\partial_{\beta}w=0,\ \ (t, x)\in\mathbb{R}_{+}\times\partial\mo.
\end{equation}
Based on \eqref{JAC}, by Lemma \ref{Timenl} and Lemma \ref{Conclusion}, we have
\begin{equation}\label{Basice1}\begin{aligned}
&\sum\limits_{0\leq |\alpha|\leq s}\|\partial\partial^{\alpha}u(t, \cdot)\|\\[2mm]
\lesssim&\sum\limits_{0\leq k\leq s-1}\|\nabla^2\partial_0^{k}u(t, \cdot)\|_{H^{s-1-k}}+\sum\limits_{0\leq k\leq s}\|\partial\partial_0^{k}u(t, \cdot)\|\\[2mm]
\lesssim&\sum\limits_{0\leq k\leq s-1}\left(\|\mathcal{Q}^{\alpha\beta}(\partial_0 v, J_1 v)\partial_{\alpha\beta}^2\partial_0^{k}u(t, \cdot)\|_{H^{s-1-k}}+\|F_k(t, \cdot)\|_{H^{s-1-k}}\right)\\[2mm]
&+\sum\limits_{0\leq k\leq s} e^{CM\varepsilon t}\left(\|\partial\partial_0^{k}u(0, \cdot)\|+\int_0^{t}\|F_k(s, \cdot)\|ds\right),
\end{aligned}
\end{equation}
where $F_k=\sum\limits_{l_1+l_2=k;\ l_1\geq 1}\partial_0^{l_1}\mathcal{Q}^{\alpha\beta}(\partial_0 v, J_1 v)\partial_{\alpha\beta}^2\partial_0^{l_2}u+\partial_0^{k}\mathcal{S}(\partial_0 v, J_1 v).
$ In addition,
\begin{equation*}
\|u(t, \cdot)\|\lesssim \int_0^t\|\partial_0 u(l, \cdot)\|dl.
\end{equation*}
Combining this with \eqref{Basice1} yields \eqref{GJ1}.\qquad
\qquad \qquad \qquad \qquad \qquad \qquad \qquad \qquad \qquad \qquad  $\square$\end{proof}

We now derive the existence of solution $u$ to the following linear Neumann-wave problem
\begin{equation}\label{conlocal1}\begin{cases}
\Box u=h^{\alpha\beta}\partial_{\alpha\beta}^2 u+G(t, x),\ \ (t, x)\in\mathbb{R}_{+}\times\mo,\\[2mm]
\partial_{\boldsymbol{\nu}}u=0,\qquad\qquad  \ (t, x)\in\mathbb{R}_{+}\times\partial\mo,\\[2mm]
(u, \partial_t u)(0, x)=\varepsilon (u_0, u_1)(x),\ \ x\in\mo.
\end{cases}
\end{equation}

\begin{thm}\label{ConLocal}  For problem \eqref{conlocal1} with $\|h(t, \cdot)\|\leq 1/2$ and
$h^{\alpha\beta}$ satisfying the ``admissible condition'' \eqref{AC}, assume that for $k\geq 2+n/2$,
\begin{equation*}
h^{\alpha\beta}(t, x)\in C^k([0, T]\times\mo),\  \ (u_0, u_1)\in H^{k}(\mo)\times H^{k-1}(\mo)
\end{equation*} and
\begin{equation}\label{5.5}
\partial_0^{j} G\in C([0, T]; H^{k-2-j}(\mo)),\ 0\leq j\leq k-2;\ \ \partial_0^{k-1}G\in L^1([0, T]; L^2(\mo)),
\end{equation}
when the corresponding compatible conditions of order $k$ are satisfied,
then there exists a unique solution $u$ to problem \eqref{conlocal1} with $\partial_0^j u\in C([0, T]; H^{k-j}(\mo))$ for $0\leq j\leq k$. Moreover,
\begin{equation}\label{conlocal2}
M_{k}(t,u)\leq C_k\left(\varepsilon\|u_0\|_{H^k}+\varepsilon\|u_1\|_{H^{k-1}}+\sup\limits_{0\leq s\leq t\atop |\alpha|\leq k-2}\|\partial^{\alpha}G(s, \cdot)\|
+\sum\limits_{0\leq j\leq\kappa-1}\int_0^{t}\|\partial_0^{j}G(s, \cdot)\|ds\right),
\end{equation}
where $C_k>0$ depends on the $C^k$ norm of $h^{\alpha\beta}$.

\end{thm}

\begin{proof} We just only give the sketch of the proof. When $\partial_0^j G\in C([0, T]; H^{k-1-j}(\mo))$ for $0\leq j\leq k-1$, then Theorem \ref{ConLocal} directly comes from Theorem 1 and Theorem 2 of \cite{IM}. In addition, estimate \eqref{conlocal2} is resulted from
the proof of Theorem 2 of \cite{IM} (one can also see Lemma \ref{Timenl} and Lemma \ref{Conclusion}).
When $G$ satisfies \eqref{5.5}, it follows from \eqref{conlocal2} and an approximate argument that Theorem \ref{ConLocal} holds.
\qquad \qquad \qquad  \qquad  \qquad  \qquad  \qquad  \qquad \qquad \qquad \qquad \qquad \qquad \qquad \qquad  $\square$\end{proof}

Next we extend Theorem \ref{ConLocal} so that it can be used to prove the local existence
of problem \eqref{Mpro}.

\begin{lem}\label{Soblocal} For problem \eqref{conlocal1} with $\|h(t, \cdot)\|\leq 1/2$
and $h^{\alpha\beta}$ satisfying the ``admissible condition'' \eqref{AC}, if $s>6+3n/2$,
\begin{equation*}
h^{\alpha\beta}\in C^{j}([0, T]; H^{s-j}(\mo)),\ \ 0\leq j\leq s,
\end{equation*}
\begin{equation*}
\partial_0^{j}G\in C([0, T]; H^{s-1-j}(\mo)),\ 0\leq j\leq s-1;\ \ \partial_0^s G\in L^1([0, T]; L^2(\mo)),
\end{equation*}
$(u_0, u_1)\in H^{s+1}(\mo)\times H^{s}(\mo)$, and the compatible conditions of order $s+1$ are satisfied, then problem \eqref{conlocal1} has a unique solution $u$ such that $\partial_0^j u\in C([0, T]; H^{s+1-j}(\mo))$ for $0\leq j\leq s+1$. Moreover,
\begin{equation}\label{conlocal3}
M_{s+1}(t,u)\leq C\left(\varepsilon\|(u_0, u_1)\|_{H^{s+1}\times H^s}+\sup\limits_{0\leq \tau\leq t\atop|\alpha|\leq s-1}\|\partial^{\alpha}G(\tau, \cdot)\|+\sum\limits_{0\leq j\leq s}\int_0^{t}\|\partial_0^{j}G(\tau, \cdot)\|d\tau\right),
\end{equation}
where $C>0$ depends on the $C^{j}([0, T]; H^{s-j}(\mo))$ norms of $h^{\alpha\beta}$ for $0\leq j\leq s$.
\end{lem}

\begin{proof} We just only give the sketch of the proof. When $s>6+3n/2$, then there exists an
integer $k\geq 2+n/2$ to be determined later, such that
\begin{equation}\label{conlocal4}
s-k>n/2,\ \ h^{\alpha\beta}(t, x)\in C^{k}([0, T]\times\mo).
\end{equation}
With the help of Theorem \ref{ConLocal}, we know that problem \eqref{conlocal1} has a unique
solution $u$ with $\partial_0^{j} u\in C([0, T]; H^{k-j}(\mo))$ for $0\leq j\leq k$.

In addition, if $\partial_0^{j}u\in C([0, T]; H^{s+1-j}(\mo))$ for $0\leq j\leq s+1$,
similar to \eqref{conlocal2}, we can obtain estimate \eqref{conlocal3}. Then only thing left is to improve the regularity of $u$, namely for $w=\partial_0^{s+1-k}u$, we should prove $\partial_0^{j}w\in C([0, T], H^{k-j}(\mo))$ for $0\leq j\leq k$.

To improve the regularity of $u$, formally, we consider the following problem for $w=\partial_0^{s+1-k}u$:
\begin{equation}\label{conlocal5}\begin{cases}
\Box w-h^{\alpha\beta}\partial_{\alpha\beta}^2w=\sum\limits_{0\leq m\leq s-k}\left(\begin{matrix}s+1-k\\ m\end{matrix}\right)\partial_0^{s+1-k-m}h^{\alpha\beta}\partial_0^{m}\partial_{\alpha\beta}^2u+\partial_0^{s+1-k}G,\\[3mm]
\qquad\qquad\qquad\qquad (t, x)\in [0, T]\times\mo,\\[2mm]
\partial_{\boldsymbol{\nu}}w=0,\qquad \qquad  \ (t, x)\in [0, T]\times \partial\mo,\\[2mm]
w(0, x)=\psi_{s+1-k}(x),\ \partial_0 w(0, x)=\psi_{s+2-k}(x),\qquad  \ x\in\mo,
\end{cases}
\end{equation}
where $u$ has the form
\begin{equation}\label{conlocal6}
u(t, x)=\psi_0(x)+t\psi_1(x)+\cdots+\frac{t^{s-k}}{(s-k)!}\psi_{s-k}(x)+\int_0^t \frac{(t-\tau)^{s-k}}{(s-k)!}w(\tau, x)d\tau.
\end{equation}
Obviously, \eqref{conlocal5} with \eqref{conlocal6} can be regarded as a nonlinear problem on $w(t,x)$.

With the argument in Theorem 9.10 in \cite{KSS2} and Theorem \ref{ConLocal}, if $u$ given by \eqref{conlocal6} is assumed to satisfy $\partial_0^{j} u\in C([0, T]; H^{s+1-j}(\mo))$ for $0\leq j\leq s+1$ in advance and for $0\leq m\leq s-k$,
\begin{equation}\label{conlocal7}
\partial_0^{s+1-k-m}h^{\alpha\beta}\partial_0^{m}\partial_{\alpha\beta}^2 u\in \bigcap\limits_{j=0}^{k-2}C^{j}([0, T]; H^{k-2-j}(\mo)),
\end{equation}
then problem \eqref{conlocal5} with \eqref{conlocal6} has a unique solution $w$ such that $\partial_0^{j}w\in C([0, T]; H^{k-j}(\mo))$ for $0\leq j\leq k$.

By the argument in Theorem 9.10 in \cite{KSS2}, \eqref{conlocal7} is fulfilled when $s-k<k-n/2$. Combining this with \eqref{conlocal4} and the restriction on $k$ in Theorem \ref{ConLocal}, we can choose $k=1+[s/2]+[n/4]$ with $s>6+3n/2$.
In this situation, with Theorem \ref{ConLocal} and the regularity argument in the proof of
Theorem 1.2 in \cite{IM}, we complete the proof of Lemma \ref{Soblocal}.\end{proof}

{\bf Proof of Theorem \ref{Lexistence}.} If we have defined $u_0(t,x)$, then the sequence $\{u_{l+1}(t,x)\}$ for $l\ge 0$
will be determined by
\begin{equation}\label{iterscheme}\begin{cases}
\Box u_{l+1}=\mathcal{Q}^{\alpha\beta}(\partial_0 u_l, J_1 u_l)\partial_{\alpha\beta}^2u_{l+1}+\mathcal{S}(\partial_0 u_l, J_1 u_{l}),\ \ (t, x)\in\mathbb{R}_{+}\times\mo,\\[2mm]
\partial_{\boldsymbol{\nu}}u_{l+1}=0,\qquad \qquad \qquad \qquad \qquad \qquad \qquad \qquad \ (t, x)\in \mathbb{R}_+\times\partial\mo,\\[2mm]
u_{l+1}(0, x)=u_0(x), \partial u_{l+1}(0, x)=u_1(x),\qquad \quad\qquad \ x\in\mo.
\end{cases}
\end{equation}
We now give the construction of $u_0$. It follows from Lemma \ref{Compatible} (A), $\psi_k\in H^{s+1-k}$ for $0\leq k\leq s+1$.
By the standard extension theorem, we know that there is a function $\Psi_k\in H^{s+1-k}(\mathbb{R}^{n})$ such that $\Psi_k\bigl|_{\mo}=\psi_k$ and $\|\Psi_k\|_{H^{s+1-k}(\mathbb{R}^{n})}\leq 2\|\psi_k\|_{s+1-k}$.
Let $(a_{lk})_{0\leq l, k\leq s+1}$ be the inverse matrix of $(i^l (k+1)^{l})_{0\leq l, k\leq s+1}$ with $i=\sqrt{-1}$,
and we set
\begin{equation*}
\hat V(t, \xi)=\sum\limits_{l, k=0}^{s+1}\exp(i(l+1)\left<\xi\right>t)a_{lk}\hat V_k(\xi)\left<\xi\right>^{-k},
\end{equation*}
where $\hat V_l$ stands for the Fourier transform of $V_l$. Define $u_0(t, x)= V(t, x)\bigl|_{\mo}$ with
$V(t, x)$ being the inverse Fourier transformation of $\hat V(t, \xi)$. By a straight verification and
Lemma \ref{Compatible} (A), one has
\begin{equation}\label{iterscheme1}
M_{s+1}[u_0]\lesssim \varepsilon \|(u_0, u_1)\|_{H^{s+1}\times H^s}\leq C(s, M)\varepsilon,
\end{equation}
and
\begin{equation}\label{iterscheme2}
\partial_0^{k}u_0(0, \cdot)=\psi_k\ (0\leq k\leq s),\ \   \partial_{\boldsymbol{\nu}}u_0=0\ \text{on}\ \partial\mo.
\end{equation}
Based on \eqref{iterscheme1}-\eqref{iterscheme2} and Lemma \ref{Compatible} (B), the existence of $\{u_l\}$ follows from Lemma \ref{Soblocal}.

We now show that for $M'=8(1+C)C(s, M)$ with $C$ given in \eqref{conlocal3}, such that for $T$ given in Lemma \ref{Soblocal},
\begin{equation}\label{Mestimate}
M_{s+1}[u_l]\leq M'\varepsilon
\end{equation}
uniformly hold for all $l$.

If we assume $M_{s+1}[u_l]\leq M'\varepsilon$, then by \eqref{GJ1} in Lemma \ref{Basice},
for problem \eqref{iterscheme}, there exists a $0<T^*<T$ independent on $l$ such that for $t\in [0, T^*]$
and $0<\varepsilon\leq \varepsilon_0$,
\begin{equation*}\label{iterestimate1}\begin{aligned}
&M_{s+1}(t, u_{l+1})\\[2mm]
\leq& Ce^{CM't\varepsilon}\left(M_{s+1}[u_0]+C(M'\varepsilon)\int_0^{t}M_{s+1}(\sigma, u_{l+1})d\sigma
+C(M'\varepsilon)M^2\varepsilon^2 t\right)\\[2mm]
&+C(M'\varepsilon)M'\varepsilon(M'\varepsilon+M_{s+1}(t,u_{l+1}))\\[2mm]
\leq&C(M'\varepsilon)\left(\int_0^{t}M_{s+1}(\sigma, u_{l+1})d\sigma+(M'\varepsilon)^2\right)+1/4 M'\varepsilon\\[2mm]
\leq&M'\varepsilon.
\end{aligned}
\end{equation*}
Here the last inequality comes from Gronwall's inequality. By induction method, we obtain \eqref{Mestimate}.

Meanwhile, we have
\begin{equation}\label{difference}\begin{cases}
\Box (u_{l+1}-u_{l})-\mathcal{Q}^{\alpha\beta}(\partial_0 u_l, J_1 u_l)\partial_{\alpha\beta}^2(u_{l+1}-u_l)=F^{l},\\[2mm]
\qquad\qquad\qquad\qquad\qquad\qquad   (t, x)\in [0, T^*]\times\mo,\\[2mm]
\partial_{\boldsymbol{\nu}}(u_{l+1}-u_l)=0,\qquad\qquad  \ (t, x)\in [0, T^*]\times\partial\mo,\\[2mm]
(u_{l+1}-u_{l})(0, x)=0, \ \partial_0 (u_{l+1}-\partial_0 u_{l})=0,\qquad   \ x\in\mo,
\end{cases}
\end{equation}
where
\begin{equation*}\begin{aligned}
F^{l}=&(\mathcal{Q}^{\alpha\beta}(\partial_0 u_{l}, J_1 u_{l})-\mathcal{Q}^{\alpha\beta}(\partial_0 u_{l-1}, J_1 u_{l-1}))\partial_{\alpha\beta}^2u_l\\[2mm]
 &+\mathcal{S}(\partial_0 u_{l}, J_1 u_{l})-\mathcal{S}(\partial_0 u_{l-1}, J_1 u_{l-1}).
\end{aligned}
\end{equation*}
Since $\mathcal{Q}^{\alpha\beta}(\partial_0 u_l, J_1 u_l)$ satisfies the ``admissible condition'' \eqref{AC} and $\|(q_{v_l}(t), \partial q_{v_l}(t))\|\lesssim M'\varepsilon$ due to \eqref{Mestimate}, it follows from the proof of Lemma \ref{Timenl} and \eqref{difference} that
\begin{equation*}
\partial_0\|\partial(u_{l+1}-u_l)(t, \cdot)\|\lesssim M'\varepsilon \|\partial(u_{l+1}-u_l)(t, \cdot)\|+\|F^{l}(t, \cdot)\|.
\end{equation*}
This implies
\begin{equation*}\begin{aligned}
\|\partial(u_{l+1}-u_{l})(t, \cdot)\|
\leq& Ce^{CM\varepsilon t}\int_0^{t}\|F^{l}(\sigma, \cdot)\|d\sigma\\[2mm]
\leq& C(M')\int_0^{t}\|\partial(u_{l}-u_{l-1})(\sigma, \cdot)\|d\sigma\\[2mm]
\leq& 2M'\frac{(C(M')T^*)^{l}}{l!},
\end{aligned}
\end{equation*}
by induction on $l$. Combining this with \eqref{Mestimate} shows that for $T^*$ suitably small, there exists a unique solution $u$ to
problem \eqref{Jpro} with $\partial_0^{j}u\in C([0, T^*]; H^{s+1-j}(\mo))\ (0\leq j\leq s+1)$.

In addition, \eqref{lexistence} follows from Lemma \ref{Compatible}, \eqref{conlocal3},
\eqref{iterscheme1}, and the induction argument. Here we omit the details. Thus
the proof of Theorem \ref{Lexistence} is completed.\qquad \qquad \qquad \qquad \qquad \quad $\square$\qed

\begin{rem}\label{Rweighted} When $(u_0, u_1)$ satisfies \eqref{intialcon}, similar to the discussions
at the beginning of Sect.5 in \cite{MC2}, we have
\begin{equation*}
\sup\limits_{0\leq t\leq 4}\sum\limits_{|\alpha|\leq 69}\|\left<x\right>^{\alpha}\partial^{\alpha}u(t, \cdot)\|\lesssim\varepsilon.
\end{equation*}
In addition, without loss of generality, we assume $T^*=4$ in Lemma \ref{Lexistence} for the simplicity of notations.
\end{rem}

\section{Global stability of 3-D compressible Chaplygin gases in exterior domain}\label{VII}

In this section, as an application of Theorem \ref{them1}, we are concerned with the global existence of a smooth  solution
to 3-D  compressible isentropic Euler system of Chaplygin gases in the exterior domain $\mo$.
The 3-D compressible isentropic  Euler system is
\begin{equation}\label{7.0}
\left\{
\begin{aligned}
&\p_t\rho+div (\rho u)=0,\\
&\p_t(\rho u)+div (\rho u \otimes u) + \nabla P=0,\\
\end{aligned}
\right.
\end{equation}
where $u=(u_1,u_2, u_3), \rho, P$ stand for the
velocity, density, pressure respectively.
For the Chaplygin gases, the equation of pressure  state (one can see \cite{CF} and so on)
is given by
$$P=P_0-\ds\f{A}{\rho},$$
where $P_0>0$ and $A>0$ are two positive constants, and $P>0$ for $\rho>A/P_0$. If $(\rho, u)\in C^1$ is a solution of \eqref{7.0} with $\rho>A/P_0$,
then \eqref{7.0} admits the following equivalent form
\begin{equation}\label{7.1}
\left\{
\begin{aligned}
&\p_t\rho+div(\rho u)=0,\\
&\p_tu+u\cdot\na u+\ds\f{\na P}{\rho}=0.\\
\end{aligned}
\right.
\end{equation}
We pose the initial-boundary data of \eqref{7.1} as follows:
\begin{equation}\label{7.2}
\left\{
\begin{aligned}
\rho(0,x)&=\bar\rho+\varepsilon\rho_0(x), \\
u(0,x)&=\varepsilon u_0(x)=\ve(u_1^0(x), u_2^0(x), u_3^0(x)),\\
\end{aligned}
\right.
\end{equation}
and
\begin{equation}\label{NBS7}
u\cdot\boldsymbol{\nu}=0,\ \ \text{on}\ \ \mathbb{R}_{+}\times\partial\mo,
\end{equation}
where  $\bar\rho>A/P_0$ is a constant, $\ve>0$
is a small parameter, and  $(\rho_0(x), u_0(x))
\in C_0^{\infty}(B(0, M)\cap\mo)$
(here and below $B(0, M)$ stands for a ball centered at the origin with a radius $M>0$,
and $\mo\cap B(0, M)\not=\emptyset$).
Moreover, $\rho(0,x)>A/P_0$, $P_0-\ds\f{A}{\bar\rho}>0$  hold and $u_0(x)$ satisfies $rot u_0(x)\equiv 0$
with $u_0(x)\cdot\boldsymbol{\nu}=0$ on $\mathbb{R}_{+}\times\partial\mo$.

Under the irrotational assumption $rot u_0(x)\equiv 0$, by \eqref{7.1}
and the finite propagation speed property of hyperbolic systems, we know
that $rot u(t,x)\equiv 0$ holds as long as the smooth solution $(\rho, u)$ of  \eqref{7.1}  exists.
Moreover $u(t,x)$ has a compact support in $x$ for any fixed $t\ge 0$. Consequently,
there exists a potential function $\varphi(t,x)$ such that $u=\nabla\varphi$.
It follows from the Bernoulli's law that
\begin{equation}\label{7.3}
\partial_t\varphi+\frac{1}{2}|\nabla\varphi|^2+h(\rho)=0,
\end{equation}
where $h(\rho)$ with $h'(\rho)=P'(\rho)/\rho$ and $h(\bar\rho)=0$ is the enthalpy of
the gases. Without loss of generality, we assume the sound speed
$c(\bar\rho)=\sqrt{P'(\bar\rho)}\equiv 1$. In this case, $h(\rho)=\f12-\f{A}{2\rho^2}$.
Substituting \eqref{7.3} and the expression of $h(\rho)$ into the equation $\partial_t\rho+div(\rho u)=0$
yields
\begin{equation*}\label{potent}
\partial_t^2\varphi+2\partial_i\varphi\partial_t\partial_i\varphi
+\partial_i\varphi\partial_j\varphi\partial_{ij}^2\varphi-(1+2\p_t\varphi+|\nabla\varphi|^2)\Delta\varphi=0.
\end{equation*}
On $\mathbb{R}_{+}\times\partial\mo$, by \eqref{NBS7}, we have the following boundary condition
\begin{equation*}\label{NB}
\partial_{\boldsymbol{\nu}}\varphi=0,\qquad \ \ (t, x)\in\mathbb{R}_{+}\times\partial\mo.
\end{equation*}
In addition, by \eqref{7.2} and $rot u_0(x)\equiv 0$, the initial data $\vp(0,x)$ and $\p_t\vp(0,x)$ can be determined as follows:
$$\vp(0,x)=\ve\int_M^{x_1}u_1^0(s,x_2,x_3)ds,\qquad \p_t\vp(0,x)=\ve\f{\rho_0(x)}{\bar\rho}+\ve^2g(x,\ve),$$
where $g(x,\ve)$ is smooth in its arguments  with compact support in $B(0, M)\cap\mo$.

Collecting those analysis above, we know that $\varphi$ satisfies the following Neumann-wave equation problem
\begin{equation}\label{pchap}\begin{cases}
\Box\varphi=Q(\partial\varphi, \partial^2\varphi),\ \ \ (t,x)\in\mathbb{R}_{+}\times\mo,\\[2mm]
\partial_{\boldsymbol{\nu}}\varphi=0,\qquad \qquad \ (t,x)\in\mathbb{R}_{+}\times\partial\mo,\\[2mm]
\varphi(0, x)=\varepsilon f_0(x),\ \ \partial_t\varphi(0, x)=\varepsilon f_1(x, \ve),\ x\in \mo,
\end{cases}
\end{equation}
where $f_0(x)$ and $f_1(x,\ve)$ are smooth in $x$ and have compact support in  $B(0, M)\cap\mo$, and
\begin{equation}\label{Polytropic}\begin{aligned}
Q(\partial\varphi, \partial^2\varphi)=&\mathcal{Q}^{\alpha\beta}(\partial\varphi)\partial_{\alpha\beta}^2\varphi\\[2mm]
=&-2\partial_i\varphi\partial_t\partial_i\varphi-\partial_i\varphi\partial_j\varphi\partial_{ij}^2\varphi
+(2\partial_t\varphi+|\nabla\varphi|^2)\Delta\varphi.
\end{aligned}
\end{equation}
Direct computation yields that the null condition holds for the equation in \eqref{pchap}.
In addition, for any smooth functions $v, w$ satisfying $\partial_{\boldsymbol{\nu}}v|_{\mathbb{R}_{+}\times\partial\mathcal{O}}=0$ and $\partial_{\boldsymbol{\nu}}w|_{\mathbb{R}_{+}\times\partial\mathcal{O}}=0$,
we have
\begin{equation*}\label{7.4}
\mathcal{Q}^{\alpha\beta}(\partial v)\boldsymbol{\nu}^{\alpha}\partial_{\beta}w
=(2\p_tv+|\na v|^2)\p_{\nu}w-\p_i v\p_iw\p_{\nu}v=0,\qquad\text{for}\ (t, x)\in\mathbb{R}_{+}\times\partial\mathcal{O},
\end{equation*}
which means that $Q(\partial\varphi, \partial^2\varphi)$ in \eqref{Polytropic} satisfies
the ``admissible condition'' \eqref{AC}. On the other hand, it is easy to check that
the compatibility conditions of arbitrary order
for the initial-boundary values in \eqref{pchap} holds. Thus as an application of Theorem \ref{them1}, we have

\begin{thm}\label{them2} For problem \eqref{7.1} together with \eqref{7.2} and \eqref{NBS7},
when $\ve>0$ is small, then there exists a global smooth solution $(\rho, u)\in C^{\infty}(\mathbb{R}_{+}\times\mo)$
with $\rho>A/P_0$ and $rot u(t,x)\equiv 0$.
\end{thm}

\begin{rem}\label{7.A}  To our best knowledge, there are only few works on
the blowup or global existence of smooth solutions to quasilinear wave equations or
multi-dimensional compressible Euler systems in exterior domain (see \cite{GP1}-\cite{GP2}
for the symmetric solutions of 3-D quasilinear wave equations, and \cite{SP2} for 2-D slightly compressible ideal flow
respectively).
As an application of Theorem \ref{them1}, the global stability of the static Chaplygin gases
outside a three dimensional obstacle $\mk$ is established in Theorem \ref{them2}.

\end{rem}

\appendix
\renewcommand{\appendixname}{Appendix\Alpha{section}}

\section{Auxiliary lemmas}\label{A}

In this appendix, we establish the following two lemmas.

\begin{lem}\label{Aux} If $v\in C^{\infty}(\mo)$ has the property $\p_{\nu}v|_{\p\mo}=0$, moreover,
$v\equiv 0$ for sufficiently large $|x|$, then if $R<t/2$ and $t\geq 1$,
\begin{equation}\label{west1}\begin{aligned}
\|\partial^2 v(t, \cdot)\|_{L^2(R/2\leq |x|\leq R)}&\lesssim \frac{1}{t}\sum\limits_{|\alpha|\leq 1}\|\Gamma^{\alpha}\partial v(t, \cdot)\|_{L^2(R/4\leq |x|\leq 2R)}+\|\Box v(t, \cdot)\|_{L^2(R/4\leq |x|\leq 2R)}\\[2mm]
&+\|\left<\cdot\right>^{-1}\partial v(t, \cdot)\|_{L^2(R/4\leq |x|\leq 2R)},
\end{aligned}
\end{equation}
\begin{equation}\label{west2}
\|\left<t-|\cdot|\right>\partial^2 v(t, \cdot)\|_{L^2(|x|\geq t/4)}\lesssim \sum\limits_{|\alpha|\leq 1}\|\Gamma^{\alpha}\partial v(t, \cdot)\|+\|\left<t+|\cdot|\right>\Box v(t, \cdot)\|,
\end{equation}
and
\begin{equation}\label{west3}
\|\partial v(t, \cdot)\|_{L^{6}(\{|x|\notin [(1-\delta)t, (1+\delta)t], |x|\geq \delta t\}\cap\mo)}
\lesssim\frac{1}{t}\left(\sum\limits_{|\alpha|\leq 1}\|\Gamma^{\alpha}\partial v(t, \cdot)\|
+\|\left<t+|\cdot|\right>\Box v(t, \cdot)\|\right).
\end{equation}

\end{lem}

\begin{proof} By Lemma 2.3 in \cite{KS4}, one has
\begin{equation}\label{west4}
\left<t-|x|\right>(|\partial\partial_0 v|+|\Delta v|)(t, x)\lesssim
\sum\limits_{|\alpha|\leq 1}|\Gamma^{\alpha}\partial v(t, x)|+\left<t+|x|\right>|\Box v(t, x)|.
\end{equation}
In addition, it follows from Lemma \ref{Elliptic} and scaling skill that
\begin{equation*}
\|\nabla^2 v(t, \cdot)\|_{L^2(R/2\le |x|\le R)}
\lesssim R^{-1}\|\partial v(t, \cdot)\|_{L^2(R/4\le |x|\le 2R)}
+\|\Delta v(t, \cdot)\|_{L^2(R/4\le |x|\le 2R)}.
\end{equation*}
Combining this estimate with \eqref{west4} shows \eqref{west1}.

Next we prove \eqref{west2}. By \cite{KS4}, for $g\in C_0^{\infty}(\mathbb{R}_{+}\times\mathbb{R}^3)$, one has
\begin{equation}\label{2.0}
\|\left<t-|\cdot|\right>\nabla^2 g(t, \cdot)\|_{L^2(\mathbb{R}^3)}\lesssim \sum\limits_{|\alpha|\leq 1}\|\Gamma^{\alpha}\partial g(t, \cdot)\|_{L^2(\mathbb{R}^3)}+\|\left<t+|\cdot|\right>\Box g\|_{L^2(\mathbb{R}^3)}.
\end{equation}
Set $g(t, x)=\rho(8x/\left<t\right>)v(t, x)$. Then one has that from \eqref{2.0} and Lemma \ref{Hardy},
\begin{equation*}
\|\left<t-|\cdot|\right>\nabla^2 v(t, \cdot)\|_{L^2(|x|\geq t/4)}\lesssim\sum\limits_{|\alpha|\leq 1}\|\Gamma^{\alpha}\partial v(t, \cdot)\|+\|\left<t+|\cdot|\right>\Box v(t, \cdot)\|.
\end{equation*}
Combining this with \eqref{west4} yields \eqref{west2}.

\eqref{west3} comes from the proof of \eqref{west1} and the fact that for $t\ge 1$
\begin{equation*}\begin{aligned}
&\|\partial v(t, \cdot)\|_{L^6(\{|x|\notin [(1-\delta)t, (1+\delta)t], |x|\geq \delta t\}\cap\mo)}\\[2mm]
\lesssim& \|\nabla\partial v(t, \cdot)\|_{L^2(\{|x|\notin [(1-\delta/2)t, (1+\delta/2)t], |x|\geq \delta t/2\}\cap\mo)}
+\|\partial v(t, \cdot)\|/t,
\end{aligned}
\end{equation*}
where the factor $1/t$ in the last term comes from the scaling skill.\qquad
\qquad\qquad\qquad\qquad\qquad$\square$\end{proof}

\begin{lem}\label{Inter} Let $v, w\in C^{\infty}(\mathbb{R}_{+}\times\mo)$ with  $\p_{\nu}v|_{\mathbb{R}_{+}\times\p\mo}=0$
and  $\p_{\nu}w|_{\mathbb{R}_{+}\times\p\mo}=0$. Moreover $v\equiv 0$ and $w\equiv 0$ for sufficiently large $|x|$. Then
for $R\ge 1$ and $t\ge 1$,
\begin{equation}\label{inter}\begin{aligned}
&\int_{\mo}\frac{\left<t-|x|\right>}{\left<t+|x|\right>\left<x\right>}|\partial^2 v(t, x)|\partial w(t, x)|dx\\[2mm]
\lesssim& \frac{1}{t}\left(\sum\limits_{|\alpha|\leq 1}\|\Gamma^{\alpha}\partial v(t, \cdot)\|+\|\left<t+|\cdot|\right>\Box u(t, \cdot)\|\right)\|\partial w(t, \cdot)/\left<\cdot\right>\|\\[2mm]
&+\sum\limits_{R=2^{k}<t/2}\left(\|\frac{1}{\left<\cdot\right>}\partial v(t, \cdot)\|_{L^2(R/4\leq |x|\leq 2R)}
+\|\frac{1}{\left<\cdot\right>^2}v(t, \cdot)\|_{L^2(R/4\leq |x|\leq 2R)}\right)\\
&\quad \times\|\frac{1}{\left<\cdot\right>}\|\partial w(t, \cdot)\|_{L^2(R/4\leq |x|\leq 2R)}.
\end{aligned}
\end{equation}
\end{lem}

\begin{proof} By Schwartz's inequality and \eqref{west1}-\eqref{west2} in Lemma \ref{Aux}, we have
\begin{equation*}\begin{aligned}
&\int_{\mo}\frac{\left<t-|x|\right>}{\left<t+|x|\right>\left<x\right>}|\partial^2 v(t, x)| \partial w(t,x)|dx\\[2mm]
\lesssim& t^{-1}\|\left<t-|\cdot|\right>\partial^2 v(t, \cdot)\|_{L^2(|x|\geq t/4)}\|
\left<\cdot\right>^{-1}\partial w(t, \cdot)\|_{L^2(|x|\geq t/4)}\\[2mm]
&+\sum\limits_{R=2^k<t/2}t^{-1}\|\left<t-|\cdot|\right>\partial^2 v(t, \cdot)\|_{L^2(R/2\leq |x|\leq R)}\|
\left<\cdot\right>^{-1}\partial w(t, \cdot)\|_{L^2(R/2\leq |x|\leq R)}\\[2mm]
\lesssim& t^{-1}\left(\sum\limits_{|\alpha|\leq 1}\|\Gamma^{\alpha}\partial v(t, \cdot)\|+\|\left<t+|\cdot|\right>\Box v(t, \cdot)\|\right)\|\left<\cdot\right>^{-1}\partial w(t, \cdot)\|\\[2mm]
&+\sum\limits_{R=2^{k}<t/2}\left(\|\left<\cdot\right>^{-1}\partial v(t, \cdot)\|_{L^2(R/4\leq |x|\leq 2R)}\|
+\|\left<\cdot\right>^{-2}v(t, \cdot)\|_{L^2(R/4\leq |x|\leq 2R)}\right)\\[2mm]
&\quad \times\|\left<\cdot\right>^{-1}\partial w(t, \cdot)\|_{L^2(R/4\leq |x|\leq 2R)}.
\end{aligned}
\end{equation*}
This yields \eqref{inter} and then  Lemma \ref{Inter} is proved.
\qquad \qquad \qquad \qquad \qquad \qquad \qquad \qquad $\square$\end{proof}

\end{document}